\documentclass{SCIYA2023enOL}
\online

\usepackage{lipsum}
\usepackage{amsfonts}
\usepackage{graphicx}
\usepackage{epstopdf}
\usepackage{algorithmic}
\usepackage{hyperref}
\usepackage{float}
\usepackage{amssymb}
\usepackage{graphicx}
\usepackage{tikz}

\newcommand{\PiCR}{\Pi_{\rm CR}}
\newcommand{\CRlam}{\lambda_{\rm CR}}
\newcommand{\CRu}{u_{\rm CR}}
\newcommand{\CRp}{p_{\rm CR}}
\newcommand{\CRuS}{u_{\rm CR}^{\lambda \Pi_h^0 u}}
\newcommand{\CRuSt}{u_{\rm CR}^{\lambda u}}
\newcommand{\CRpS}{p_{\rm CR}^{\lambda \Pi_h^0 u}}
\newcommand{\CRpSt}{p_{\rm CR}^{\lambda  u}}

\newcommand{\VCR}{V_{\rm CR}}
\newcommand{\ZCR}{Z_{\rm CR}}

\newcommand{\CR}{{\rm CR}}

\newcommand{\PiECR}{\Pi_{\rm ECR}}
\newcommand{\ECRlam}{\lambda_{\rm ECR}}
\newcommand{\ECRu}{u_{\rm ECR}}
\newcommand{\ECRp}{p_{\rm ECR}}
\newcommand{\ECRuS}{u_{\rm ECR}^{\lambda \Pi_h^0 u}}
\newcommand{\ECRuSt}{u_{\rm ECR}^{\lambda u}}
\newcommand{\ECRpS}{p_{\rm ECR}^{\lambda \Pi_h^0 u}}
\newcommand{\ECRpSt}{p_{\rm ECR}^{\lambda u}}

\newcommand{\VECR}{V_{\rm ECR}}
\newcommand{\ZECR}{Z_{\rm ECR}}

\newcommand{\ECR}{{\rm ECR}}

\newcommand{\PiRT}{\Pi_{\rm RT}}

\newcommand{\RTuS}{u_{\rm RT}^{\lambda u}}

\newcommand{\RTsig}{\sigma_{\rm RT}}
\newcommand{\RTsigS}{\sigma_{\rm RT}^{\lambda  u}}
\newcommand{\RTsigSt}{\sigma_{\rm RT}^{\lambda  u}}
\newcommand{\URT}{U_{\rm RT}}
\newcommand{\SIGRT}{\Sigma_{\rm RT}}
\newcommand{\phiRT}{\phi_{\rm RT}}

\newcommand{\RT}{{\rm RT}}

\newcommand{\bmx}{\boldsymbol{x}}
\newcommand{\bmn}{\boldsymbol{n}}
\newcommand{\bmt}{\boldsymbol{t}}
\newcommand{\bmp}{\boldsymbol{p}}
\newcommand{\bmm}{\boldsymbol{m}}
\newcommand{\MK}{\boldsymbol{M}_K}

\newcommand{\Kh}{\boldsymbol{K}_h}

\newcommand{\bmu}{u}
\newcommand{\bmv}{v}
\newcommand{\bmw}{w}
\newcommand{\bmV}{V}
\newcommand{\bmVh}{V_h}
\newcommand{\Q}{{\rm Q}}
\newcommand{\Qh}{\Q_h}
\newcommand{\Zh}{Z_h}
\newcommand{\bmZ}{Z}
\newcommand{\ba}{\mathbf{A}}
\newcommand{\bc}{\mathbf{C}}
\newcommand{\bd}{\mathbf{D}}
\newcommand{\be}{\mathbf{E}}
\newcommand{\Bf}{\mathbf{F}}

\newcommand{\bsa}{\mathbf{a}}
\newcommand{\bid}{\mathbf{I}_2}

\newcommand{\piCRi}{\phi_{\rm CR}^i}
\newcommand{\etaij}{\eta_{\rm RT}^{ij}}

\newcommand{\gamij}{\gamma_{\rm RT}^{ij}}
\newcommand{\zetai}{\zeta_{\rm CR}^{i}}

\newcommand{\Div}{{\nabla\cdot\,}}
\newcommand{\Divh}{{\nabla_h\cdot\,}}

\newcommand{\Tr}{{\rm tr}}
\newcommand{\Dev}{{\rm dev\,}}

\newcommand{\X}{{\rm X}}
\newcommand{\T}{G}

\newcommand{\R}{\mathbb{R}}
\newcommand{\Rd}{\R^2}
\newcommand{\Rdd}{\R^{2\times 2}}
\newcommand{\Nd}{\mathbb{N}^2}
\newcommand{\Th}{\mathcal{T}_h}
\newcommand{\Fh}{\mathcal{E}_h}
\newcommand{\cE}{\mathcal{E}}
\newcommand{\F}{e}
\newcommand{\Pk}{\mathbb{P}_{k}}
\newcommand{\bbP}{\mathbb{P}}

\newcommand{\Hdivd}{H({\rm div},\Omega,\R^{2})}

\newcommand{\bmtau}{\tau}
\newcommand{\Pih}{\Pi_h^0}
\newcommand{\PiK}{\Pi_K^0}
\newcommand{\Piu}{\Pi_h^u}
\newcommand{\PiuCR}{\Pi_{\rm CR}^u}
\newcommand{\PiuECR}{\Pi_{\rm ECR}^u}
\newcommand{\Pip}{\Pi_h^p}
\newcommand{\balpha}{{{\alpha}}}
\newcommand{\dx}{\,{\rm dx}}
\newcommand{\ds}{\,{\rm ds}}

\title{An example}

\begin{document}

\ensubject{fdsfd}

\Year{2022}
\Month{January}%
\Vol{65}
\No{1}
\BeginPage{1} %
\DOI{ }

\title[]{Supercloseness and asymptotic analysis of the Crouzeix-Raviart and enriched Crouzeix-Raviart elements  for the Stokes problem}
{Supercloseness and asymptotic analysis of the Crouzeix-Raviart and enriched Crouzeix-Raviart elements  for the Stokes problem}

\author[1,2]{Wei Chen}{{1901110037@pku.edu.cn}}
\author[1,2]{Hao Han}{{hanhao@pku.edu.cn}}
\author[3,$\ast$]{Limin Ma}{{limin18@whu.edu.cn} \thanks{The first and the second authors are supported by NSFC project 12288101. The third author is supported by NSFC project 12301523 and the Fundamental Research Funds for the Central Universities 413000117. }} 

\AuthorMark{ }

\AuthorCitation{ }

\address[1]{LMAM and School of Mathematical Sciences, Peking University,  Beijing {\rm100871}, People's Republic of China} 
\address[2]{Chongqing Research Institute of Big Data, Peking University, Chongqing {\rm 401121}, People's Republic of China} 
\address[3]{School of Mathematics and Statistics, Wuhan University, Wuhan {\rm 430072}, People's Republic of China} 

%

\abstract{
For the Crouzeix-Raviart and enriched Crouzeix-Raviart elements,  asymptotic expansions of eigenvalues of the Stokes operator are derived by establishing two pseudostress interpolations, which admit a full one-order supercloseness with respect to the numerical velocity and the pressure, respectively. The design of these interpolations overcomes the difficulty caused by the lack of supercloseness of the canonical interpolations for the two nonconforming elements, and leads to an intrinsic and concise asymptotic analysis of numerical eigenvalues, which proves an optimal superconvergence of eigenvalues by the extrapolation algorithm. Meanwhile,  an optimal superconvergence of postprocessed approximations for the Stokes equation is proved by use of this supercloseness. Finally, numerical experiments are tested to verify the theoretical results.
} 



\maketitle

\section{Introduction}
Extrapolation algorithm is an important tool to improve the accuracy of exact solutions, and attracts lots of interest, see for instance \cite{Blum1990Finite,Chen2007Asymptotic,Ding1990quadrature,Jia2010Approximation,MR2126582,Lin2008New,Lin2009new,Lin2007Finite,Lin1984Asymptotic,lin1999high,Lin2009Asymptotic,lin2010new,Lin2011Extrapolation,Luo2002High}, and the references therein.
Asymptotic expansions provide the theoretical foundation for extrapolation algorithms, and the supercloseness of certain interpolation operator is the key ingredient in the classical asymptotic analysis.
There exist many work on the supercloseness of conforming finite elements and mixed finite elements, see \cite{MR1987941,Brandts1994Superconvergence,MR1755693,1995High,MR1083049,MR2772093}.
However, for nonconforming elements, the numerical solutions may not admit the desired supercloseness with respect to the canonical interpolation because of the consistency error.
This leads to a substantial difficulty in analyzing the supercloseness and therefore the asymptotic expansions of eigenvalues.
A classic tool to resolve this difficulty is to design a corrected interpolation, to which the numerical solution admits the supercloseness.
However, the design of such interpolations is elaborate and most of the existing results focus on rectangular-like meshes, see \cite{Chen2013Superconvergence,MR2470146,MR1485999} for more details.
For the widely used Crouzeix-Raviart (CR for short hereinafter) and enriched Crouzeix–Raviart (ECR for short hereinafter) element on triangular meshes, it is until recently that the first asymptotic expansion of eigenvalues for the Laplace operator is analyzed in \cite{MR4350533} by employing the equivalence between CR, ECR and the lowest order Raviart–Thomas (RT for short hereinafte) elements. 
The application of the supercloseness of the mixed RT element, instead of proving the supercloseness of the nonconforming elements in consideration, leads to a highly technical theoretical analysis in  \cite{MR4350533}.
Similar results for the Morley element can be found in \cite{Huang2008new,Lin2009new,MR4456710}.
However, the supercloseness of the nonconforming CR and ECR elements is still an issue.



For the asymptotic analysis of the Stokes eigenvalues, the mixed finite element methods by the Bernadi–Raugel element and the $Q_2-P_1$ element, the stream function-vorticity-pressure method by bilinear finite element and nonconforming $Q_1^{\rm rot}$ and $EQ_1^{\rm rot}$ is proved on rectangular meshes \cite{MR2400623,MR2476018,MR2197324}. 
In the cases of nonconforming triangular CR and ECR element, a direct application of the equivalence between the CR, ECR and RT element in \cite{MR3066804,MR3328190},  as the case in \cite{MR4350533} for second-order elliptic problems, 
requires not only the asymptotic analysis of velocity, but also that of pressure. 
However, the canonical interpolation of pressure, just as that of velocity, does not admit the critical  supercloseness, which brings in additional challenges to the asymptotic analysis.

In this paper, the key supercloseness of the nonconforming CR and ECR element for Stokes problem is established. 
To overcome the difficulty caused by the consistency error, two pseudostress interpolations are designed by employing the equivalence between the CR element, the ECR element and the RT element for the Stokes equation in \cite{MR3066804,MR3328190}.
An optimal full one-order supercloseness is proved for these pseudostress interpolations with respect to both the approximate velocity and pressure on uniform triangular meshes, which is the first supercloseness result for the nonconforming CR and ECR elements.
The design of pseudostress interpolations shows that the weak continuity of finite element space and the interaction within different variables play an important role in the supercloseness analysis. 
Moreover, a direct application of the supercloseness and a simple postprocess technique leads to a global $\mathcal{O}(h^2)$-superconvergence result of CR element and ECR element for the Stokes equation.

Therefore, an asymptotic analysis of the nonconforming CR and ECR elements for the Stokes eigenvalue problem is conducted to prove the optimal convergence of eigenvalues by the extrapolation algorithm.
Thanks to the supercloseness for both the velocity and the pressure, the asymptotic analysis in this paper only requires the expansion of the interpolation error of velocity. Compared to the analysis in \cite{MR4350533, MR4456710} for the Laplace and the biharmonic operator, the analysis of the Stokes operator in this paper is much more simple and intrinsic because of the supercloseness, and can also be extended to simplify the analysis significantly for the two cases in \cite{MR4350533, MR4456710}. In the new analysis, there is no need to conduct an optimal $\mathcal{O}(h^4)$  analysis of many consistency error terms, or establish asymptotic expansion of the discrete pressure, while both of which are required if the technique in \cite{MR4350533, MR4456710} is applied.

The remaining paper is organized as follows.
 Some notations are introduced in Section \ref{notation}. 
The supercloseness and superconvergence of  the CR element and ECR element are analyzed for the Stokes equation in Section \ref{SuperStokesProblem}. 
The optimal asymptotic expansions of eigenvalues by the CR element and ECR element and therefore the optimal convergence by the extrapolation algorithm are investigated in Section \ref{asymptotic}. 
Some numerical tests are conducted to validate the theoretical results in Section \ref{Numerical}.

\subsection{Notations and preliminaries}\label{notation}

Suppose that $\Omega\subset \Rd$ is a convex polygonal domain and the partition $\Th$ of domain $\Omega$ is assumed to be uniform in the sense that any two adjacent triangles form a parallelogram. For any element  $K\in\Th$, let $|K|$  and $h_K$ be the area and the diameter  of $K$, respectively, and  $h=\max_{K\in\Th}h_K$. Denote the set of all interior edges and boundary edges of $\Th$ by $\Fh^i$ and $\Fh^b$, respectively, and $\Fh=\Fh^i\cup\Fh^b$. Let $|e|$ be  the length of any edge $e$, $\bmn_e$ be the unit normal vector of the edge $e=K_1\cap K_2$ pointing from the element $K_1$ with smaller global index to the one with lager index.

For any element $K$ with vertices $\bmp_i=\left(p_{i1},p_{i2}\right)^T, 1\leq i \leq 3$ oriented counterclockwise, denote the corresponding barycentric coordinates, the perpendicular heights, the unit outward normal vectors and the centroid of the element by $\{\psi_i\}_{i=1}^{3}$,  $\{h_i\}_{i=1}^{3}$, $\{\bmn_i\}_{i=1}^{3}$ and  $\MK=\left(M_1,M_2\right)^T$, respectively.
Let $\{e_i\}_{i=1}^{3}$ and $\{\bmm_i\}_{i=1}^{3}$ be the edges of element $K$, and the midpoints of edges $\{e_i\}_{i=1}^{3}$, respectively. There hold
$
h_i|e_i|=2|K|$ and $\nabla\psi_i=-\frac{\bmn_i}{h_i}
$
in \cite{MR2398767}. 
Let the finite-dimensional space $\X$ be either $\R$, $\Rd$ or $\Rdd$. For any $a,b\in\Rd$, define $a\cdot b=a^Tb\in\R$ and $a\otimes b=ab^T\in\Rdd$. For any $\tau\in\Rdd$, let $\tau_i$ be the $i$-th row of the matrix $\tau$.
For any $\sigma,\tau \in\Rdd$, define the Euclid product  by $\sigma:\tau=\sum_{i,j=1}^2\sigma_{ij}\tau_{ij}$. For any matrixes $a, b, c, d\in \Rdd$, define a tensor $a\circ b\in \R^{2\times 2\times 2 \times 2}$ with
$
(a\circ b)_{ijkl}=a_{ij}b_{kl}.
$
Let $ \bid$ be the unit matrix in $\Rdd$, and denote the trace and the deviatoric part of any $\sigma\in\Rdd$ by
$\Tr\ \sigma:=\sigma:\bid$ and $\Dev \sigma:=\sigma-\frac{1}{2}(\Tr \sigma) \bid,
$
respectively. 

Given a nonnegative integer $k$, let $L^2(\T,\X)$, $W^{k, \infty}(\T,\X), H^k(\T, \X)$ be the usual Sobolev space consisting of functions taking values in $\X$ on a region $\T\subset\Rd$, with the norms and semi-norm denoted by $\|\cdot\|_{k,\infty, \T}$, $\|\cdot\|_{k, \T}$ and $|\cdot|_{k, \T}$, respectively.
The standard $L^2$ inner products on the domain $\Omega$ and any region $\T$ are denoted by $(\cdot, \cdot)$ and $(\cdot, \cdot)_{\T}$, respectively.
Define the mean-zero subspace of $L^2(\Omega,\R)$ by
$L^2_0(\Omega,\R)=\{v\in L^2(\Omega,\R): (v,1)=0\}.$
Let $\nabla_h$ and $\Divh$ denote the piecewise gradient and the piecewise divergence with respect to the triangulation $\Th$, respectively.
For any point $\bmx=(x_1,x_2)^T$ and any index $\balpha=\left(\alpha_1,\alpha_2\right)^T\in\Nd,$ denote
$
\bmx^{\balpha}=x_1^{\alpha_1}x_2^{\alpha_2}$, $ \vert{\balpha}\vert= \alpha_1+\alpha_2$, $ \balpha!=\alpha_1!\alpha_2!$, $  D^\balpha v=\partial_{x_1}^{\alpha_1}\partial_{x_2}^{\alpha_2}v,
$
where $\mathbb{N}$ is the set of nonnegative integers.
Let $\Pk(\T)$ be the set of polynomials on $\T$ with degrees not larger than $k$, and $\Pk(\T,\X)$ be the set of polynomials on $\X$ with each entry belongs to $\Pk(\T)$.
For ease of presentation, the symbol $A\lesssim B$ will be used to denote that $A\leq CB$ where $C$ is a positive constant independent of mesh size $h$.

\section{Supercloseness and superconvergence for the nonconforming elements}\label{SuperStokesProblem}

For both the CR element and the ECR element, two pseudostress interpolations are designed and the discrete velocity and pressure are proved to be superclose to these two interpolations with $\mathcal{O}(h^2)$, respectively. 
Furthermore, the superconvergence of both nonconforming elements are proved by use of this supercloseness property and a  simple postprocessing algorithm.

\subsection{Nonconforming and mixed element methods for the Stokes equation}
Consider the Stokes problem with source term $f\in L^2(\Omega,\Rd)$
\begin{equation}\label{source}
\left\{
\begin{aligned}
	- \Delta \bmu^f - \nabla p^f &= f & \mbox{ in }&\Omega,\\
	\Div \bmu^f &= 0                  & \mbox{ in }&\Omega,\\
	\bmu^f &= 0     & \mbox{ on }&\partial\Omega.\\
\end{aligned}
\right.
\end{equation}
The weak form of the velocity-pressure formulation of \eqref{source} seeks $(\bmu^f,p^f)\in \bmV\times Q$  such that
$$
\left\{
\begin{aligned}
(\nabla u^f, \nabla v)+(\nabla\cdot v, p^f)&=(f, v), \quad &\mbox{ for any } v \in \bmV, \\
(\nabla\cdot u^f, q)&=0, \quad &\mbox{ for any } q \in Q,
\end{aligned}
\right.
$$  
where $\bmV:=H^1_0(\Omega,\Rd)$ and $Q:=L^2_0(\Omega,\R)$. Given the triangulation $\Th$, let $\bmVh$ be a nonconforming finite element approximation of $\bmV$, and~$\Qh$ be the piecewise constant subspace of $Q$.
Consider two nonconforming elements: the CR element \cite{MR0343661} and the ECR element \cite{hu2014lower,lin2010stokes,MR3163260} as follows. 

	 The CR element space over $\Th$  is defined in \cite{MR0343661} by
		$$
		\begin{aligned}
			\CR(\Th):= \{v \in L^2(\Omega, \R):&~v|_K \in \bbP_1(K, \R)\mbox{ for any }K\in\Th,\\
			     &\int_\F [v] \ds=0, \mbox{ for any } e \in \Fh^i, \int_\F  v \ds=0, \mbox{ for any } e \in \Fh^b\}.
		\end{aligned}
		$$
		Denote the approximation of the velocity space $V$ by the CR element by $\VCR:=\CR(\Th,\Rd)$, and the corresponding canonical interpolation operator $\PiCR:\bmV \to \VCR$  by
		$\int_\F  \PiCR\bmv \ds = \int_\F v\ds$ for any $\F\in\Fh.$
	         The CR element method of \eqref{source} finds $(\CRu^f, \CRp^f)\in\VCR\times \Qh$ such that
		\begin{equation}\label{CRwithPih0f}
		\left\{
		\begin{aligned}
		(\nabla_h \CRu^f, \nabla_h v_h)+(\Divh\bmv_h, \CRp^f)&=(f, v_h), \quad &\mbox{ for any } v_h \in \VCR, \\
		(\Divh\CRu^f, q_h)&=0, \quad &\mbox{ for any } q_h \in \Qh.
		\end{aligned}
		\right.
		\end{equation}	
	The ECR element space over $\Th$ is defined in \cite{MR3163260} by
		$$
		\begin{aligned}
			\ECR(\Th):= \{v \in &L^2(\Omega, \R):~v|_K \in \ECR(K, \R)\mbox{ for any }K\in\Th,\\
			     &\int_\F [v] \ds=0, \mbox{ for any } e \in \Fh^i, \int_\F  v \ds=0, \mbox{ for any } e \in \Fh^b\},
		\end{aligned}
		$$
		where $\ECR(K, \R):=\bbP_1(K,\R)+\mbox{span}\{x_1^2+x_2^2\}$.
		Denote the approximation of the velocity space $V$ by the ECR element by $\VECR:=\ECR(\Th,\Rd)$ and the corresponding canonical interpolation operator $\PiECR:\bmV\to\VECR$  satisfies
		$
			\int_\F \PiECR\bmv \ds=\int_\F v \ds$ and $\int_K\PiECR\bmv \dx=\int_Kv \dx$ for any $\F\in\Fh,~ K\in\Th.
		$
		The ECR element method of \eqref{source} finds $(\ECRu^f,\ECRp^f)\in\VECR\times \Qh$ such 		that
		\begin{equation}\label{ECRwithPih0f}
		\left\{
		\begin{aligned}
		(\nabla_h \ECRu^f, \nabla_h v_h)+(\Divh\bmv_h,\ECRp^f)&=(f, v_h), \quad &\mbox{ for any } v_h \in \VECR, \\
		(\Divh\ECRu^f, q_h)&=0, \quad &\mbox{ for any } q_h \in \Qh.
		\end{aligned}
		\right.
		\end{equation}

Define the pseudostress $\sigma^f=\nabla \bmu^f + p^f\bid$, it follows that
\begin{equation}\label{usigRelation}
p^f=\frac12 \Tr \sigma^f,\quad \nabla\bmu^f=\Dev\sigma^f.
\end{equation}
This yields the equivalent  pseudostress-velocity formulation of~\eqref{source}, which seeks $(\sigma^f,\bmu^f)\in\Sigma\times U$ such that
\begin{equation}\label{mixCon}
\left\{
\begin{aligned}
(\Dev\sigma^f,\Dev\tau) + (\Div\tau, \bmu^f) &= 0,\qquad &\forall \tau\in\Sigma,\\
(\Div\sigma^f,\bmv)&=-(f,\bmv),\qquad &\forall \bmv\in U,
\end{aligned}
\right.
\end{equation}
with $U=L^2(\Omega,\Rd)$, 
$\Hdivd=\{v\in L^2(\Omega,\Rd): \nabla\cdot v\in L^2(\Omega,\R)\}$ and
$$
\Sigma:=\left\{\bmtau\in L^2(\Omega,\Rdd): \bmtau_i\in\Hdivd,~i=1,2,\text{ and }\int_\Omega\Tr(\bmtau)\dx=0\right\}.
$$
Consider the lowest order RT element \cite{MR0483555} for the Stokes equation \eqref{source}, where the shape function space  is
$
\RT(K,\Rd):=\bbP_0(K,\Rd) + \bmx\bbP_0(K,\R),
$
and the corresponding finite element space
$
\RT(\Th):=\left\{\bmtau\in \Hdivd:\bmtau|_K\in\RT(K,\Rd),\ \forall K\in\Th\right\}.
$
Define
$$
\begin{aligned}
\RT(K,\Rdd):&=\{\bmtau\in L^2(K,\Rdd):~\bmtau_i\in\RT(K,\Rd), \text{ for }i=1,2\},
\end{aligned}
$$
and the approximation space of pseudostress  by the RT element
$
\SIGRT(\Th):=\{\bmtau\in\Sigma:~\bmtau|_K\in\RT(K,\Rdd)\text{ for all }K\in\Th\}.
$
The canonical interpolation operator 
of the RT element in \cite{MR0771029,MR1075159} is defined by
$\int_\F (\PiRT\bmtau-\bmtau)\bmn_e\ds=0$ for any edge $e\in \cE_h$.
For any $\tau\in\Sigma$,
\begin{equation}\label{RTdof}
\PiRT\tau|_K=\displaystyle\sum_{i=1}^{3}\frac{1}{2\left|{K}\right|}\int_{\F_i}\tau\bmn_i \ds\otimes\left(\bmx-\bmp_i\right).
\end{equation}
The piecewise constant space $\URT$ is used to approximate the velocity and get a stable pair of spaces.
The RT element method of the Stokes equation \eqref{source}  seeks
$(\sigma_{\rm RT}^f, \bmu_{\rm RT}^f)\in\SIGRT(\Th)\times\URT$ such that
\begin{equation}\label{RTPi0u}
\left\{
\begin{aligned}
\left(\Dev\sigma_{\rm RT}^f, \Dev\tau_h\right)+  (\Div\tau_h,\bmu_{\rm RT}^f)&=0, \quad &\mbox{ for any } \tau_h \in \SIGRT(\Th), \\
(\Div \sigma_{\rm RT}^f, v_h)&=-(f, v_h), \quad &\mbox{ for any } v_h \in \URT,
\end{aligned}
\right.
\end{equation}
By the definition of $\PiRT$, it holds that $\int_{K} \Div(I-\PiRT) \tau \dx =0$ for any $ K\in\Th$.
Then, by \eqref{mixCon} and \eqref{RTPi0u},
\begin{equation}\label{propertxih}
\Div (\PiRT\sigma^f-\RTsig^f)=0,\quad \mbox{and}\quad (\PiRT\sigma^f-\RTsig^f)|_K\in \bbP_0(K,\Rdd).
\end{equation}
As analyzed in  \cite[Theorem 3.1]{MR4434148}, the approximate pseudostress by the mixed formulation \eqref{RTPi0u} admits the following superconvergence  property on uniform triangulations.
\begin{lemma}\label{supeRTlm}
Let $(\sigma^f, \bmu^f)$ and $(\sigma_{\rm RT}^f, \bmu_{\rm RT}^f)$ be the solutions of Problems \eqref{mixCon} and \eqref{RTPi0u}, respectively, with  $(\sigma^f,u^f)\in W^{2,\infty}(\Omega,\Rdd)\times H^2(\Omega,\Rd)$.
It holds on uniform triangulations with sufficiently small~$h$
\begin{equation}\label{supeRT}
\left\|\PiRT\sigma^f-\sigma_{\rm RT}^f\right\|_{0,\Omega} \lesssim h^2|\ln h|^{\frac{1}{2}}(\|{\sigma^f}\|_{W_{\infty}^2(\Omega)}+\|{u^f}\|_{2,\Omega}).
\end{equation}
\end{lemma}

\subsection{Supercloseness for the nonconforming elements} \label{superclose}
To begin with, define the $L^2$ projection operators $\PiK: L^2(K,X)\to\bbP_0(K,X)$ and $\Pih:L^2(\Omega,X)\to\bbP_0(\Th,X)$ by
$
\PiK w = \frac1K\int_Kw\dx$ and $\Pih w|_K=\PiK w,
$
respectively. It holds that 
$
\|u-\PiK u\|_{0,K}\lesssim h|u|_{1,K}
$
for any $u\in H^1(K,X)$.
For the pseudostress $\sigma \in \Sigma$, define the following pseudostress interpolations
\begin{equation}\label{def:inter}
\begin{aligned}
\Pip\sigma:&=\frac12\Pih\Tr(\PiRT\sigma), 
\\
\PiuCR\sigma :&=\Pih \Dev \PiRT\sigma= \Pih\PiRT\sigma-(\Pip\sigma)\bid,
\\
\PiuECR\sigma :&=\Pih \Dev \PiRT\sigma + (I-\Pih)\PiRT\sigma= \PiRT\sigma-(\Pip\sigma)\bid.
\end{aligned}
\end{equation} 


There exists the following equivalence between the RT solution and the CR and the ECR solutions of the Stokes equation in \cite{MR3066804} and \cite{MR3328190}, respectively.
\begin{lemma}\label{RERTandCRECR}
Let $(\CRu^f, \CRp^f)$, $(\ECRu^f, \ECRp^f)$, and $(\sigma_{\rm RT}^f, \bmu_{\rm RT}^f)$ be the solutions of problems \eqref{CRwithPih0f}, \eqref{ECRwithPih0f} and ~\eqref{RTPi0u} with piecewise constant  $f$, respectively.
It holds that
\begin{align}
\label{equivenceofRTandCR}
\sigma_{\rm RT}^f|_K&=\nabla_h \CRu^f|_K+\CRp^f \bid-\frac{f}{2} \otimes \left(\bmx-\MK\right),\quad \bmx\in K,\\
\label{equivenceofRTandECR}
\sigma_{\rm RT}^f|_K&=\nabla_h \ECRu^f|_K+\ECRp^f \bid.
\end{align}
\end{lemma}

By use of these equivalence, a full first-order supercloseness with respect to the pseudostress interpolations for both velocity and pressure is proved as follows.

\begin{theorem}\label{th:superclose}
Let $(\sigma^f,u^f)$, $(\sigma_{\rm RT}^f, \bmu_{\rm RT}^f)$, $(\CRu^f, \CRp^f)$ and $(\ECRu^f, \ECRp^f)$ be the solution of problems \eqref{mixCon}, \eqref{RTPi0u}, \eqref{CRwithPih0f} and \eqref{ECRwithPih0f} with source term $f$, respectively. Assume that  $(\sigma^f,u^f)\in W^{2,\infty}(\Omega,\Rdd)\times H^2(\Omega,\Rd)$ and the triangulation is uniform with sufficiently small $h$. Then,
\begin{equation}
\begin{aligned}\label{super:pu}
\|\CRp^f- \Pip\sigma^f\|_{0,\Omega} + \|\ECRp^f- \Pip\sigma^f\|_{0,\Omega}
+& \|\nabla_h\CRu^f-\PiuCR\sigma^f\|_{0,\Omega} 
\\
+ \|\nabla_h\ECRu^f-\PiuECR\sigma^f\|_{0,\Omega}\lesssim &h^2|\ln h|^{\frac{1}{2}}(\|{\sigma^f}\|_{2,\infty, \Omega}+\|{u^f}\|_{2,\Omega}).
\end{aligned}
\end{equation}
\end{theorem}
\begin{proof}
It follows from \eqref{RTPi0u} and the definition of $\URT$ that $\sigma_{\rm RT}^f=\sigma_{\rm RT}^{\Pi_h^0f}$. Applying~$\Pih$ to both sides of \eqref{equivenceofRTandCR} and \eqref{equivenceofRTandECR},  the fact that  $\CRp^f$ and $\ECRp^f$ are piecewise constant indicates that
$$
\Pi_h^0 \sigma_{\rm RT}^f= \Pi_h^0\sigma_{\rm RT}^{\Pi_h^0f}=\nabla_h \CRu^{\Pi_h^0f}+\CRp^{\Pi_h^0f} \bid=\Pi_h^0\nabla_h \ECRu^{\Pi_h^0f}+\ECRp^{\Pi_h^0f} \bid.
$$
The discrete divergence-free property in \eqref{CRwithPih0f} and \eqref{ECRwithPih0f} implies  that
$
\Divh\CRu^{\Pi_h^0f}=\Pi_h^0\Divh\ECRu^{\Pi_h^0f}=0.
$
These two equations and \eqref{equivenceofRTandECR} indicate that
 \begin{equation}\label{CRp0withsigma}
\begin{aligned}
&\CRp^{\Pi_h^0f}=\ECRp^{\Pi_h^0f}
=\frac12  \Pi_h^0  {\Tr}(\sigma_{\rm RT}^f),
\\ 
&\nabla_h\CRu^{\Pi_h^0f}=\Pih\Dev\sigma_{\rm RT}^f,
\qquad
\nabla_h\ECRu^{\Pi_h^0f}=\sigma_{\rm RT}^f-\frac12  \Pi_h^0  {\Tr}(\sigma_{\rm RT}^f)\bid.
\end{aligned}
 \end{equation} 
Let the difference $e_h^u=\CRu^f - \CRu^{\Pi_h^0f} $ and $e_h^p=\CRp^{f}-\CRp^{\Pi_h^0f}$. Then,
\begin{equation}\label{erreq}
\left\{
\begin{aligned}
(\nabla_h e_h^u, \nabla_h v_h)+(\Divh\bmv_h, e_h^p)&=((I-\Pi_h^0)f, v_h), \quad &\mbox{ for any } v_h \in \VCR, \\
(\Divh e_h^u, q_h)&=0, \quad &\mbox{ for any } q_h \in \Qh,
\end{aligned}
\right.
\end{equation}
where $(e_h^u, e_h^p)$ can be viewed as the CR approximation to the solution $(e_u,e_p)$ of problem \eqref{source} with source term $(I-\Pi_h^0)f$. Then,
$
\|\nabla_h(e_h^u-e_u)\|_{0,\Omega} + \|e_h^p - e_p\|_{0,\Omega}\lesssim h(\|e_u\|_{2,\Omega} + \|e_p\|_{1,\Omega})
$
if the domain is convex (see \cite[(5.11),(6.12)]{MR0343661}). It follows from the inverse inequality that $\|e_h^u\|_{1,\Omega}+\|e_h^p\|_{0,\Omega}\lesssim  \|e_u\|_{1,\Omega} + \|e_p\|_{0,\Omega}$. The wellposedness in \cite{galdi1994stokes}  as follows
$$
\|e_u\|_{1,\Omega} + \|e_p\|_{0,\Omega}\lesssim \|(I-\Pi_h^0)f\|_{-1,\Omega}=\sup_{v\in H_0^1(\Omega)}\frac{|((I-\Pi_h^0)f,(I-\Pi_h^0)v)|}{\|v\|_{1,\Omega}}\lesssim h^2\|f\|_{1,\Omega},
$$
indicates that $\|e_h^u\|_{1,\Omega}+\|e_h^p\|_{0,\Omega}\lesssim h^2$.
A combination of this,  \eqref{supeRT} and  \eqref{CRp0withsigma} yields
$$
\begin{aligned} 
\|\CRp^f - \Pip\sigma^f\|_{0,\Omega}&=\|e_h^p 
+ \frac12\Pi_h^0\Tr (\sigma_{\rm RT}^f-\PiRT \sigma^f)
\|_{0,\Omega} 
\lesssim h^2|\ln h|^{\frac{1}{2}},\\
\|\nabla_h\CRu^f - \Piu\sigma^f\|_{0,\Omega}&=\|
\nabla_h e_h^u 
+ \Pi_h^0\Dev(\sigma_{\rm RT}^f-\PiRT \sigma^f)\|_{0,\Omega}
\lesssim h^2|\ln h|^{\frac{1}{2}},
\end{aligned}
$$
which completes the proof for the first and third estimates in \eqref{super:pu} of the CR element.
The same argument applies for the ECR element
and completes the proof.
\end{proof} 

\begin{remark}
For the Poisson equation, a similar interpolation to $\PiuCR$ can also be defined such that the numerical solution by the CR element admits a full first-order supercloseness with respect to the new interpolation of the exact solution.
\end{remark}

\subsection{Superconvergence for the nonconforming finite elements}

By the supercloseness above, approximate pressure and velocity with a first-order superconvergence can be constructed by applying an appropriate postprocessing algorithm.
For any piecewise function $\textbf{q}$ taking values in $\X$,
define  $\Kh\textbf{q}\in\CR(\Th,\X)$ as follows.
\begin{definition}\label{Def:R}
1.For each interior edge $e\in\cE_h^i$, the elements $K_e^1$ and $K_e^2$ are the pair of elements sharing $e$. Then the value of $K_h \textbf{q}$ at the midpoint $\textbf{m}$ of $e$ is
$$
K_h \textbf{q}(\textbf{m})={1\over 2}(\textbf{q}|_{K_e^1}(\textbf{m})+\textbf{q}|_{K_e^2}(\textbf{m})).
$$

2.For each boundary edge $e\in\cE_h^b$, let $K$ be the element having $e$ as an edge, and $K'$ be an element sharing an edge $e'\in\cE_h^i$ with $K$. Let $e''$ denote the edge of $K'$ that does not intersect with $e$, and $\textbf{m}$, $\textbf{m}'$ and $\textbf{m}''$ be the midpoints of the edges $e$, $e'$ and $e''$, respectively. Then the value of $K_h \textbf{q}$ at the point $\textbf{m}$ is
$$
K_h \textbf{q}(\textbf{m})=2K_h \textbf{q}(\textbf{m}') - K_h \textbf{q}(\textbf{m}'').
$$
\begin{center}
\begin{tikzpicture}[xscale=2,yscale=2]
\draw[-] (-0.5,0) -- (2,0);
\draw[-] (0,0) -- (0.5,1);
\draw[-] (0.5,1) -- (2,1);
\draw[-] (0.5,1) -- (1.5,0);
\draw[-] (2,1) -- (1.5,0);
\node[below, right] at (1,0.5) {\textbf{m}'};
\node[above] at (1.25,1) {\textbf{m}''};
\node[below] at (0.75,0) {\textbf{m}};
\node at (0.7,0.4) {K};
\node at (1.4,0.75) {K'};
\node at (1,0) {e};
\node at (1.4,0.2) {e'};
\node at (1.7,1) {e''};
\node at (0.3,-0.1) {$\partial \Omega $};
\fill(1,0.5) circle(0.5pt);
\fill(1.25,1) circle(0.5pt);
\fill(0.75,0) circle(0.5pt);
\end{tikzpicture}
\end{center}
\end{definition}

As pointed out in \cite[Theorem 5.1]{Brandts1994Superconvergence} for the Poisson equation, $\Kh\PiRT\sigma$ is a higher order approximation of $\sigma$ than $\PiRT\sigma$ itself. The lemma below shows that for an even less continuous function $\Pih\PiRT\sigma$, an application of the postprocessing operator $K_h$ 
can still achieve a higher order approximation of $\sigma$, which is vital in analyzing the superconvergence for the Stokes equation.

\begin{lemma}\label{Khpro}
Suppose $\sigma\in H^2(\Omega,\Rdd)$, it holds on uniform triangulations
$$
\|\sigma-\Kh\Pih\PiRT\sigma\|_{0,\Omega}\lesssim h^2|\sigma|_{2,\Omega}.
$$
\end{lemma}
\begin{proof}
For any element $K$ with edges $\{e_i\}_{i=1}^3$ , denote the triangle sharing edge~$e_i$ with $K$ by $K_i$. Let $\tilde K$ be the union of triangles sharing vertices with $K$. For any $r\in\bbP_1(\tilde{K},\Rdd)$,  it is proved in \cite[Lemma 3.1]{Brandts1994Superconvergence} that 
$
\int_{K\cup K_i}(r-\PiRT r)\dx=0,
$
which indicates that 
$
\int_{K\cup K_i}(r-\Pih\PiRT r)\dx=0.
$
On the other hand,  
$$
\int_{K\cup K_i}(r-\Pih\PiRT r)\dx = (|K|+|K_i|)\left(r(\bmm_i) - \frac12(\Pi_K^0\PiRT r + \Pi_{K_i}^0\PiRT r)\right),
$$
which indicates that $r(\bmm_i)=\Kh\Pih\PiRT r(\bmm_i)$ when $\bmm_i$ is the midpoint of an interior edge. If $\bmm$ is the midpoint of a boundary edge, $\bmm'$ and $\bmm''$ are interior points, $\Kh\Pih\PiRT r(\bmm)=2\Kh\Pih\PiRT r(\bmm') - \Kh\Pih\PiRT r(\bmm'')=2r(\bmm')-r(\bmm'')=r(\bmm)$ since $r$ is linear in $\tilde K$.
Thus, for any $r\in\bbP_1(\tilde{K},\Rdd)$, $\Kh\Pih\PiRT r=r$ on $K$, namely $\sigma - \Kh\Pih\PiRT\sigma=(I-\Kh\Pih\PiRT)(\sigma -r)$ with $\sigma\in H^2(\Omega,\Rdd)$. Moreover,
\begin{equation}
\label{proKh}
\|\sigma - \Kh\Pih\PiRT\sigma\|_{0,K} \lesssim h \inf_{r\in \bbP_1(\tilde{K},\Rdd)}\|(I-\Kh\Pih\PiRT)(\sigma -r)\|_{0,\infty,K}.
\end{equation}
If $\bmm$ is the midpoint of an interior edge $e=\overline{K_e^1}\cap \overline{K_e^2}$, denote $\omega_{\bmm}=K_e^1\cup K_e^2$. Then,
$$
|\Kh\Pih\PiRT \sigma(\bmm)|=\left|\frac{\Pi_{K_e^1}^0\PiRT\sigma+\Pi_{K_e^2}^0\PiRT\sigma}{2}\right|
\le \|\PiRT\sigma\|_{0,\infty,\omega_{\bmm}}.
$$
If $\bmm$ is the midpoint of a boundary edge, denote $\omega_{\bmm}=\omega_{\bmm'}\cup\omega_{\bmm''}$. Then, 
$$
|\Kh\Pih\PiRT \sigma(\bmm)|=\left|2\Kh\Pih\PiRT \sigma(\bmm')-\Kh\Pih\PiRT \sigma(\bmm'') \right|\le 3\|\PiRT\sigma\|_{0,\infty,\omega_{\bmm}}.
$$
It follows that
\begin{equation}\label{pi0bdd}
\|\Kh\Pih\PiRT \sigma\|_{0,\infty,K}
\leq 3\max_{1\leq i\leq 3}|\Kh\Pih\PiRT \sigma(\bmm_i)|
\leq 9\|\PiRT\sigma\|_{0,\infty,\tilde{K}}.
\end{equation}
As proved in \cite[Theorem 5.1]{Brandts1994Superconvergence}, $\PiRT$ is a bounded operator in $L^\infty$-norm,
it holds that $
\|\Kh\Pih\PiRT \sigma\|_{0,\infty,K}\lesssim  \|\sigma\|_{0,\infty,\tilde{K}}.
$
A combination of this and  \eqref{proKh} indicates that
$
\|\sigma - \Kh\Pih\PiRT \sigma\|_{0,K}\lesssim h\|\sigma-r\|_{0,\infty,\tilde K} 
$ for any $r\in\bbP_1(\tilde{K},\Rdd)$.
By the interpolation theory in Sobolev spaces (\cite[Chapter 3]{MR0520174}),  
\begin{equation}
\label{proKh1}
\|\sigma-\Kh\Pih\PiRT\sigma\|_{0,K}\lesssim h^2\|\sigma\|_{2,\tilde{K}}.
\end{equation}
By squaring \eqref{proKh1} and summing over all triangles $K\in\Th$, we complete the proof.  
\end{proof}

Thanks to the supercloseness \eqref{supeRT} and Lemma \ref{Khpro}, it holds the following superconvergence results for the Stokes equation.
\begin{theorem}\label{th:superconvergence} 
Let $(u^f,p^f)$, $(\CRu^f, \CRp^f)$, $(\ECRu^f, \ECRp^f)$, $(\sigma^f,u^f)$ and $(\sigma_{\rm RT}^f, \bmu_{\rm RT}^f)$ be the solution of problems \eqref{source}, \eqref{CRwithPih0f}, \eqref{ECRwithPih0f}, \eqref{mixCon} and \eqref{RTPi0u} with source term $f$, respectively. Assume that  $(\sigma^f,u^f)\in W^{2,\infty}(\Omega,\Rdd)\times H^2(\Omega,\Rd)$ and the triangulation is uniform with sufficiently small $h$. Then,
$$
\begin{aligned}
&\|\sigma^f-\Kh\Pih\sigma_{\rm RT}^f\|_{0,\Omega}
+\|\nabla u^f-\Kh\nabla_h\CRu^f\|_{0,\Omega}
+\|\nabla u^f-\Kh\nabla_h\ECRu^f\|_{0,\Omega}
\\
&\quad+\|p^f-\Kh\CRp^f\|_{0,\Omega} +\|p^f-\Kh\ECRp^f\|_{0,\Omega} \lesssim h^2|\ln h|^{\frac{1}{2}}(\|{\sigma^f}\|_{W_{\infty}^2(\Omega)}+\|{u^f}\|_{2,\Omega}).
\end{aligned}
$$
\end{theorem}

\begin{proof}
The same argument for \eqref{pi0bdd} proves
$
\|\Kh\Pih(\PiRT\sigma^f-\sigma_{\rm RT}^f)\|_{0,\infty,K}\leq 9\|\Pih(\PiRT\sigma^f-\sigma_{\rm RT}^f)\|_{0,\infty,\tilde{K}}.
$
On a uniform triangulation $\Th$,
$$
\|\Kh \Pih(\PiRT\sigma^f-\sigma_{\rm RT}^f)\|_{0,\Omega}
\leq (\sum_{K\in \Th}81\|\Pih(\PiRT\sigma^f-\sigma_{\rm RT}^f)\|_{0,\tilde{K}}^2)^\frac12.
$$
It follows from this and the triangular inequality that
\begin{equation}\label{SCSWZ}
\|\sigma^f-\Kh\Pih\sigma_{\rm RT}^f\|_{0,\Omega}\lesssim\|\sigma^f-\Kh\Pih\PiRT\sigma^f\|_{0,\Omega}+\|\Pih(\PiRT\sigma^f-\sigma_{\rm RT}^f)\|_{0,\Omega}.
\end{equation}
A combination of this, the superconvergence property in Lemma \ref{supeRTlm} and  \ref{Khpro}  leads to  the proof for the first term. By \eqref{usigRelation},
$$
\nabla u^f-\Kh\nabla_h\CRu^f=(\sigma^f-\Kh\Pih\PiRT\sigma^f)-\frac12\Tr(\sigma^f-\Kh\Pih\PiRT\sigma^f)\bid
$$
which, together with  Lemma  \ref{Khpro},  yields the estimates for the CR element. A similar procedure without the projection $\Pih$ as applied in \cite[Theorem 5.1]{Brandts1994Superconvergence} leads to the estimate for the ECR element and completes the proof.
\end{proof}

\section{Asymptotic analysis for the Stokes eigenvalue problem} \label{asymptotic}
In this section, an asymptotic expansion of eigenvalues of the Stokes problem is established to prove an optimal superconvergence rate of the extrapolation algorithm.

The Stokes eigenvalue problem  seeks $(\lambda,\bmu,p)$ with $\Vert{\bmu}\Vert_{0,\Omega} = 1$, $\int_\Omega p \dx=0$ and
\begin{equation}\label{sto1}
\left\{
\begin{aligned}
	- \Delta \bmu - \nabla p &= \lambda \bmu & \mbox{ in }&\Omega,\\
	\Div \bmu &= 0                  & \mbox{ in }&\Omega,\\
	\bmu &= 0     & \mbox{ on }&\partial\Omega.\\
\end{aligned}
\right.
\end{equation}
The weak form of the corresponding velocity-pressure formulation of \eqref{sto1} seeks $(\lambda,\bmu,p)\in\R\times \bmV\times Q$ with $\left\|{\bmu}\right\|_{0,\Omega}=1$ such that
\begin{equation}\label{sto2}
\left\{
\begin{aligned}
	a(\bmu, \bmv)+b(\bmv, p) &=\lambda(\bmu,\bmv) &\forall \bmv\in\bmV,\\
	b(\bmu, q)&=0 & \forall q\in \Q,
\end{aligned}
\right.
\end{equation}
where  the bilinear forms
$a(\bmw,\bmv):=\int_\Omega\nabla\bmw : \nabla\bmv \dx$ and $b(\bmv, q):=\int_\Omega\Div\bmv~q\dx.$
Let $\bmZ:=\left\{\bmv\in\bmV:~\Div\bmv=0\right\}$. Weak form \eqref{sto2} indicates that $(\lambda,\bmu)\in\R\times\bmZ$ satisfies
\begin{equation}\label{stoz1}
a(\bmu,\bmv)=\lambda(\bmu,\bmv) \qquad \forall \bmv\in\bmZ.
\end{equation}
The eigenvalue problems \eqref{sto2} and \eqref{stoz1} have the same eigenvalue sequence
$0<\lambda_1 \leq \lambda_2 \leq \lambda_3 \leq \cdots \nearrow+\infty,$
and the corresponding eigenfunctions of  \eqref{sto2} read
$\left(\bmu_1,p_1\right),\left(\bmu_2,p_2\right), \left(\bmu_3,p_3\right),\cdots$
with
$\left(\bmu_i,\bmu_j\right)=\delta_{ij}\mbox{ with } \delta_{ij}
$
is the Kronecker symbol. See \cite{MR1115240} for more details. 
For an eigenvalue $\lambda$ of \eqref{sto2}, define the eigenfunction space  
$
M(\lambda)=\left\{(\bmw, q) \in \bmV \times \Q :~(\lambda,\bmw, q) \mbox{ is a solution of \eqref{sto2} }\right\}.
$

The  finite element approximation of \eqref{sto2} is to find $\left(\lambda_h,\bmu_h, p_h\right)\in\R\times\bmVh\times \Qh$ with $\left\|{\bmu_h}\right\|_{0,\Omega}=1$ such that
\begin{equation}\label{sto3}
\left\{
\begin{aligned}
    a_h\left(\bmu_h, \bmv_h\right)+b_h\left(\bmv_h, p_h\right) &=\lambda_h(\bmu_h,\bmv_h) &\forall \bmv_h\in\bmVh,\\
	b_h\left(\bmu_h, q_h\right)&=0  &\forall q_h\in \Qh,
\end{aligned}
\right.
\end{equation}
where 
$\displaystyle
a_h(\bmw_h,\bmv_h):=\sum_{K\in\Th}\int_K\nabla_h\bmw_h:\nabla_h\bmv_h\dx$, $
b_h(\bmv_h,q_h):=\sum_{K\in\Th}\int_K\Divh\bmv_h q_h\dx.
$
Then the discrete eigenpair $(\lambda_h,\bmu_h)\in\R\times\Zh$ satisfies that
\begin{equation}\label{stoz2}
	a_h(\bmu_h,\bmv_h)=\lambda_h(\bmu_h,\bmv_h)\qquad \forall \bmv_h\in\Zh,
\end{equation}
where $\Zh:=\left\{\bmv_h\in\bmVh:~b_h(\bmv_h, p_h)=0,\mbox{ for all } p_h\in\Qh\right\}$ with $N=\mbox{dim}\Zh$. Denote  $(\lambda_h,\bmu_h, p_h, Z_h)$ in the CR element space $\VCR$ by $(\CRlam, \CRu,\CRp,\ZCR)$, and those in  the ECR element space $\VECR$ by $(\ECRlam, \ECRu,\ECRp,\ZECR)$.
A subscript $i$ is added to distinguish the approximate eigenpairs related to different eigenvalues. For example, $\left(\bmu_{\mathrm{CR},i},p_{\mathrm{CR},i}\right)$ is the discrete solution by the CR element related to the $i$-th eigenvalue $\lambda_{\mathrm{CR},i}$.
It follows from the theory of nonconforming eigenvalue approximations in \cite{jia2014posterior,lovadina2009posteriori} that if the domain is convex and $M(\lambda)\subset \left(V\cap H^2(\Omega)\right)\times (\Q \cap H^1(\Omega))$, there exists $\left(u, p\right)\in M(\lambda)$ with $\lambda=\lambda_i$ such that
\begin{equation}\label{CRest1}
|\lambda-\lambda_h| + \|\bmu-\bmu_h\|_{0,\Omega} + \|\bmu-\Pi_h\bmu\|_{0,\Omega} + h\|\nabla_h(\bmu-\bmu_h)\|_{0,\Omega}\lesssim h^2(\|u\|_{2,\Omega} + \|p\|_{1,\Omega}),
\end{equation}
where $\left(\lambda_h,\bmu_h, p_h,\Pi_h\right)=\left(\lambda_{\mathrm{CR},i}, \bmu_{\mathrm{CR},i}, p_{\mathrm{CR},i},\PiCR\right)$ or $\left(\lambda_{\mathrm{ECR},i}, \bmu_{\mathrm{ECR},i}, p_{\mathrm{ECR},i},\PiECR\right)$ are the solutions and the canonical interpolations of \eqref{sto3} by the CR element or the ECR element, respectively.
Whenever there is no ambiguity, $\left(\lambda, \bmu, p\right)$ defined this way is called the corresponding eigenpair to $\left(\lambda_{\mathrm{CR}\!,i} , \bmu_{\mathrm{CR}\!,i}, p_{\mathrm{CR}\!,i}\right)$ and  $\left(\lambda_{\mathrm{ECR}\!,i}, \bmu_{\mathrm{ECR}\!,i}, p_{\mathrm{ECR}\!,i} \right)$ of problem \eqref{sto3} if the estimate \eqref{CRest1} holds.

There holds  the following commuting property for both nonconforming elements,
\begin{align}
\int_K \nabla_h\left(\bmw-\Pi_h \bmw\right) : \nabla_h \bmv_h \dx=0 &\quad \mbox{ for any } \bmw \in \bmV, \bmv_h \in \bmVh,\label{commuting}
\\
\int_K\Divh(\bmw-\Pi_h\bmw)q_h~ \dx=0& \quad \mbox{ for any } q_h\in \Qh,\label{interp2}
\end{align}
where $(\Pi_h,\bmVh)=(\PiCR,\VCR)$ or $(\PiECR,\VECR)$.
See \cite{MR0343661, MR3163260} for more details.
Note that the eigenpair $(\lambda, \bmu)\in\R\times\bmZ$ satisfies  \eqref{stoz1}, and the discrete eigenpair $(\lambda_h, \bmu_h)\in\R\times\Zh$ by both elements satisfies  \eqref{stoz2}. By the commuting property \eqref{commuting} and the analysis in  \cite{MR4078805, MR4350533}, there exists the following expansion of  approximate eigenvalues
\begin{equation}\label{CRlamexap1}
	\lambda-\lambda_h = \left\|{\nabla_h(\bmu-\bmu_h)}\right\|^2_{0,\Omega}-2\lambda_h(\bmu-\Pi_h\bmu,\bmu_h)-\lambda_h\left\|{\bmu-\bmu_h}\right\|^2_{0,\Omega},
\end{equation}
where the commuting property \eqref{interp2} is required here to ensure that $\Pi_h\bmu\in \Zh$.


\subsection{Error expansions for eigenvalues}
%
%
%
%

By the commuting property \eqref{commuting} of~$\PiCR$, a similar analysis to that in \cite[Theorem 3.2]{MR4350533} yields
	\begin{equation}\label{ExpandlambdaminuslambdaCR}
	\lambda-\CRlam=\left\|\nabla_h(u-\CRu)\right\|^2_{0,\Omega}-2 \lambda(u-\PiCR u, u)+\mathcal{O}(h^4).
	\end{equation}
The key for an optimal asymptotic analysis of eigenvalues is to establishing the expansion of $\left\|\nabla_h(u-\CRu)\right\|_{0,\Omega}$ with accuracy $\mathcal{O}(h^4)$.

As presented in the lemma below, the difference of approximate solutions with different source terms converges at a higher rate, where the omitted analysis is similar to that in \cite{MR4350533} and \cite{MR4456710}.

\begin{lemma}
Let $(\CRuS,\CRpS)$ and $(\CRuSt,\CRpSt)$ be the solutions of  problem \eqref{CRwithPih0f} with $f=\lambda \Pi_h^0\bmu$ and $\lambda \bmu$, respectively, $(\ECRuS , \ECRpS)$ and $(\ECRuSt,\ECRpSt)$ be the solutions of  problem \eqref{ECRwithPih0f} with $f=\lambda \Pi_h^0\bmu$ and $\lambda \bmu$, respectively, $(\CRlam, \CRu,\CRp)$ and $(\ECRlam,\ECRu,\ECRp)$ be the solutions of problem \eqref{sto3} by $\rm CR$ element and $\rm ECR$ element, respectively. It holds that
\begin{align}\label{supCRu}
\|\nabla_h (\CRuS-\CRuSt) \|_{0,\Omega} + \left\|{\nabla_h\left(\CRuS-\CRu\right)}\right\|_{0,\Omega}  \lesssim h^2(\left\|{u}\right\|_{2,\Omega}+\|p\|_{1,\Omega}),
\\
\label{supcloseECR}
\|\nabla_h (\ECRuS-\ECRuSt) \|_{0,\Omega} + \|\nabla_h (\ECRuS-\ECRu ) \|_{0,\Omega}\lesssim h^2(\left\|{u}\right\|_{2,\Omega}+\|p\|_{1,\Omega}).
\end{align}
\end{lemma}

By use of the supercloseness \eqref{super:pu} of the CR element,   the discrete eigenvalues can be expressed in terms of the interpolation error as proved in the lemma below, which is vital in the asymptotic analysis of eigenvalue.

\begin{lemma}\label{th22}
Let $(\lambda, \bmu, p)$ and $(\CRlam, \CRu, \CRp)$ be the solutions of problem \eqref{sto1} and \eqref{sto3} by the $\rm CR$ element, respectively, and $\sigma = \nabla \bmu + p\bid$. Assume that $(\sigma,u)\in W^{2,\infty}(\Omega,\Rdd)\times H^2(\Omega,\Rd)$. On a uniform triangulation with sufficiently small $h$, 
\begin{equation*}
     \lambda-\CRlam 
	=\left\|\nabla u - \PiuCR\sigma\right\|^2_{0,\Omega}
-2 \lambda(u-\PiCR u, u) +\mathcal{O}(h^4|\ln h| (\|{\sigma}\|^2_{2,\infty, \Omega}+\|{u}\|^2_{2,\Omega}) ).
\end{equation*} 
\end{lemma}
\begin{proof}
Consider the first term on the right-hand side of \eqref{ExpandlambdaminuslambdaCR}. 
Note that
\begin{equation}\label{CRexpand1}
\begin{aligned}
\nabla_h (u-\CRu)
	=&(\nabla u-\PiuCR\sigma) + (\PiuCR \sigma - \nabla_h\CRuS) + \nabla_h(\CRuS-\CRu),
\end{aligned}
\end{equation}   
where $\PiuCR \sigma - \nabla_h\CRuS=\Dev\xi_h$ by \eqref{CRp0withsigma} with $\xi_h=\PiRT\sigma - \RTsigS$ is piecewise constant by~\eqref{propertxih}.
By the supercloseness in Theorem \ref{th:superclose} and \eqref{supCRu}, the $L^2$ norm of the last two terms   converge at the rate two.
A substitution of \eqref{CRexpand1} into \eqref{ExpandlambdaminuslambdaCR} leads to
\begin{equation}
\begin{aligned}\label{laminuscrlacexapn}
     \lambda-\CRlam
	=&\left\|\nabla u - \PiuCR\sigma\right\|^2_{0,\Omega}-2 \lambda(u-\PiCR u, u)
	+2\left(\Dev (I -\PiRT)\sigma, \Dev\xi_h\right) 
	\\
	&+2 (\nabla u-\PiuCR\sigma, \nabla_h(\CRuS-\CRu) )  +\mathcal{O}(h^4),
\end{aligned}
\end{equation}
where
$\nabla u - \PiuCR\sigma=\Dev\Pih(I -\PiRT)\sigma + \Dev(I -\Pih)\sigma$. By \eqref{mixCon}, \eqref{RTPi0u}, the superconvergence~\eqref{supeRT} and the divergence free property of $\xi_h$ in \eqref{propertxih},
\begin{equation}\label{decoRT1}
\begin{aligned}
\left|\left(\Dev (I -\PiRT)\sigma, \Dev\xi_h\right)\right| \le  &\left(\Dev (\sigma - \RTsig), \Dev\xi_h\right) + Ch^4|\ln h|
\\
\le &\left|(\nabla\cdot \xi_h, u-\RTuS)\right|+ Ch^4|\ln h|\lesssim h^4|\ln h|.
\end{aligned}
\end{equation} 
Since $\Dev\sigma=\nabla u$ and $\PiuCR\sigma=\Pih\Dev \PiRT\sigma$,  a combination of  the commuting property in \eqref{commuting}, the supercloseness property in Theorem \ref{th:superclose} and \eqref{supCRu} yields 
$$
\begin{aligned}
& (\nabla u-\PiuCR\sigma, \nabla_h(\CRuS-\CRu) )
\\
=& (\nabla_h (u-\CRuS) + (\nabla_h\CRuS-\PiuCR\sigma), \nabla_h(\CRuS-\CRu) )
\\
=&(\nabla_h s_h, \nabla_h(\CRuS-\CRu)) + \mathcal{O}(h^4|\ln h|^{\frac12}),
\end{aligned}
$$
where $s_h=\PiCR u - \CRuS$.
It follows from eigenvalue problem \eqref{sto3} by the $\CR$ element, the source problem \eqref{CRwithPih0f} and the fact that $\Divh s_h=0$ by \eqref{interp2} that
\begin{equation*}\label{proICR12}
(\nabla_h s_h, \nabla_h(\CRuS-\CRu))
= \CRlam (s_h,(\Pi_h^0-I )\CRu) + (s_h, \Pi_h^0(\lambda u-\CRlam\CRu)).
\end{equation*}
By
$
\lambda u-\CRlam\CRu=\lambda (u-\CRu)+(\lambda-\CRlam)\CRu
$
and $\left\|{s_h}\right\|_{0,\Omega}+h\left|{\nabla_h s_h}\right|_{0,\Omega}\lesssim h^2\left|{u}\right|_{2,\Omega}$,
$$
(\nabla u-\PiuCR\sigma, \nabla_h(\CRuS-\CRu))= -\CRlam ( (I-\Pi_h^0 )s_h,  (I-\Pi_h^0 )\CRu ) + \mathcal{O}\left(h^4\left|{\ln h}\right|\right).
$$
A substitution of \cite[Lemma 3.11]{MR4350533} into the identity above proves
\begin{equation}\label{decoRT2}
|(\nabla u-\PiuCR\sigma, \nabla_h(\CRuS-\CRu))|\lesssim h^4|\ln h|.
\end{equation}
A combination of \eqref{laminuscrlacexapn}, \eqref{decoRT1} and \eqref{decoRT2}
completes the proof.
\end{proof}

\begin{lemma}\label{th21}
Let $(\lambda, \bmu, p)$ and $(\ECRlam, \ECRu, \ECRp)$ be the solutions of~\eqref{sto1} and \eqref{sto3} by the ECR element, respectively, and $\sigma = \nabla \bmu + p\bid$. Assume that $(\sigma,u)\in W^{2,\infty}(\Omega,\Rdd)\times H^3(\Omega,\Rd)$. On a uniform triangulation with sufficiently small $h$, 
\begin{equation}\label{ExpandECR}
\begin{aligned}
	\lambda-\ECRlam=&\left\|\nabla u - \PiuECR\sigma\right\|^2_{0,\Omega} 
+\mathcal{O}(h^4|\ln h| (\|{\sigma}\|^2_{2,\infty, \Omega}+\|{u}\|^2_{3,\Omega}) .
	\end{aligned}
\end{equation} 
\end{lemma}
\begin{proof}
Note that
\begin{equation}\label{ECRexpand1}
\nabla_h (u-\ECRu)
=(\nabla u-\PiuECR\sigma) + (\PiuECR \sigma - \nabla_h\ECRuS) + \nabla_h(\ECRuS-\ECRu),
\end{equation}   
where  it follows from  \eqref{def:inter} and \eqref{CRp0withsigma} with piecewise constant  $\xi_h=\PiRT\sigma - \RTsigS$ that
$$\PiuECR \sigma - \nabla_h\ECRuS=\Dev\xi_h.$$
Since $\nabla u - \PiuECR\sigma=\Dev\Pih(I -\PiRT)\sigma + (I -\Pih)(\Dev\sigma - \PiRT\sigma)$, a combination of \eqref{decoRT1} and the fact that $\xi_h$ is piecewise constant leads to
\begin{equation}\label{ECRexpandest1}
\begin{aligned}
|(\nabla u-\PiuECR\sigma, \PiuECR \sigma - \nabla_h\ECRuS)|
= |(\Dev(I -\PiRT)\sigma,  \Dev\xi_h)| 
\lesssim h^4|\ln h|.
\end{aligned}
\end{equation} 
Thanks to the commuting property in \eqref{commuting}, the supercloseness property in Theorem \ref{th:superclose} and \eqref{supcloseECR}, it holds that 
\begin{equation}\label{ECRexpandest2}
\begin{aligned}
&\left(\nabla u-\PiuECR\sigma, \nabla_h(\ECRuS-\ECRu)\right)
\\
=&\left(\nabla_h (u-\ECRuS) + (\nabla_h\ECRuS-\PiuECR\sigma), \nabla_h(\ECRuS-\ECRu)\right)
\\
=&(\nabla_h w_h, \nabla_h(\ECRuS-\ECRu)) + \mathcal{O}(h^4|\ln h|^{\frac12}),
\end{aligned}
\end{equation} 
where $w_h=\PiECR u - \ECRuS$. By \eqref{sto3}, commuting property \eqref{interp2} and~$\Div\bmu=0$, it holds that $\Pi_h^0\Divh w_h=0$. Thus, it follows from \eqref{ECRwithPih0f} with source term $f=\lambda\Pi_h^0 u$, \eqref{sto3}, \eqref{CRest1},  the orthogonal property of  $\Pi_h^0$, and \cite[Lemma 4.3]{MR4350533} that 
\begin{equation}\label{computenablaECRandPiECR}
\begin{aligned}
&|(\nabla_h w_h,\nabla_h(\ECRuS-\ECRu))|
\\
=&|\ECRlam((I-\Pih)\ECRu,(I-\Pih)w_h)| + \mathcal{O}(h^4))
\lesssim h^4(\|u\|^2_{2,\Omega}+\|p\|^2_{1,\Omega}).
\end{aligned}
\end{equation}  
By the supercloseness in Theorem \ref{th:superclose} and \eqref{supcloseECR}, the $L^2$ norm of the last two terms in \eqref{ECRexpand1} converge at the rate two. 
A substitution of  \eqref{ECRexpand1}, \eqref{ECRexpandest1}, \eqref{ECRexpandest2} and  \eqref{computenablaECRandPiECR} into \eqref{CRlamexap1} leads to
$$
     \lambda-\ECRlam
	=\left\|\nabla u - \PiuECR\sigma\right\|^2_{0,\Omega}-2 \lambda(u-\PiECR u, u)
	+\mathcal{O}(h^4),
$$
which, together with 
$
|(u-\PiECR u, u)|=|(u-\PiECR u, u-\Pih u)|\lesssim h^4|u|_{3,\Omega}^2
$
by the estimate in \cite[Lemma 4.2]{MR4350533},
 completes the proof.
\end{proof}

According to Lemma \ref{th22} and \ref{th21}, the asymptotic analysis for problem~\eqref{sto3} by the CR element and the ECR element requires the asymptotic expansion of the interpolation error terms
$\left\|\nabla u - \PiuCR\sigma\right\|^2_{0,\Omega}$, $(u-\PiCR u, u),$  and
$\left\|\nabla u - \PiuECR\sigma\right\|^2_{0,\Omega}$. 

\subsection{Asymptotic expansion of interpolation error terms}\label{sec:21}
According to \cite[Lemma 3.8]{MR4350533}, it holds that $|(u-\PiCR u, u-\Pih u)|\lesssim h^4|u|^2_{3,\Omega}$ if the triangulation is uniform.  By $\nabla u=\Dev\sigma$,  
\begin{equation}
\label{Intmain}
\begin{aligned}
\|\nabla u\!-\!\PiuCR\sigma\|^2_{0,\Omega}\!&=\!\|(\Dev\!-\!\PiuCR)\sigma\|_{0,\Omega}^2,
\\
\|\nabla u\!-\!\PiuECR\sigma\|^2_{0,\Omega}\!&=\!\|(\Dev\!-\!\PiuECR)\sigma\|_{0,\Omega}^2.
\\
(u-\PiCR u, u) &= (u-\PiCR u, \Pih u) + \mathcal{O}(h^4|u|^2_{3,\Omega}).
\end{aligned}
\end{equation}
To begin with the analysis of the right hand sides of  \eqref{Intmain}, define the following eight short-hand notations with $i=1$ and $2$,
$$
\begin{aligned}
	\phiRT^i(\bmx) &=r_i\otimes \left(
		x_2-M_2, \, 
		x_1-M_1
	\right)^T, \ 
	\quad \phiRT^{2+i}(\bmx) = r_i\otimes\left(
		x_2-M_2, \,
		-(x_1-M_1)
	\right)^T, \\
    \phiRT^{4+i}(\bmx) &= r_i\otimes \left(
        x_1-M_1,\,
         -(x_2-M_2)
   \right)^T,
    \quad \phiRT^{6+i}(\bmx) = r_i\otimes \left(
        x_1-M_1, \,
         x_2-M_2
   \right)^T,
    \end{aligned}
$$
where $r_1=(1,0)^T$ and $r_2=(0,1)^T$,
and also eight differential operators for any matrix-valued function $\sigma=(\sigma_{jk})$  with $i=1$ and $2$
$$
\begin{aligned}
\operatorname{D}_{i}\sigma=\frac{1}{2}(\partial_{x_2}\sigma_{i1}+\partial_{x_1}\sigma_{i2}),\quad
\operatorname{D}_{2+i}\sigma=\frac{1}{2}(\partial_{x_2}\sigma_{i1}-\partial_{x_1}\sigma_{i2}),
\\
\operatorname{D}_{4+i}\sigma=\frac{1}{2}(\partial_{x_1}\sigma_{i1}-\partial_{x_2}\sigma_{i2}), \quad
\operatorname{D}_{6+i}\sigma=\frac{1}{2}(\partial_{x_1}\sigma_{i1}+\partial_{x_2}\sigma_{i2}).
\end{aligned}
$$
For the CR element,  define quadratic functions
$
\piCRi=(2\psi_{i-1}-1)(2\psi_{i+1}-1)-\frac23\psi_i+\frac13$ with $1\leq i\leq 3, 
$
where the barycentric coordinates $\{\psi_i\}_{i=1}^3$  with indices modulo 3. 
For each element~$K$, define three types of constants
\begin{equation}\label{gammij}
\begin{aligned}
\gamij:&=\frac{1}{h^2\vert{K}\vert}((\Dev-\PiuCR)\phiRT^i, (\Dev-\PiuCR)\phiRT^j)_K,\quad 1\leq i,j\leq 8,\\
\etaij :&= \frac{1}{h^2\vert{K}\vert}((\Dev-\PiuECR)\phiRT^i, (\Dev-\PiuECR)\phiRT^j)_K,\quad 1\leq i,j\leq 8,\\
\zetai:&=\frac{|e_i|^2}{h^2|K|}\int_K\piCRi \dx,\quad 1\leq i \leq 3.
\end{aligned}
\end{equation}


\begin{lemma}\label{P1expand}
Constants $\gamij$, $\etaij$ and $\zetai$ in \eqref{gammij} take the same value on different elements on uniform triangulations, and are independent of mesh size $h$. Furthermore, for any $w\in \mathbb{P}_1(K,\Rdd)$, $v\in \bbP_2(K,\Rd)$, it holds that
	\begin{equation}\label{P1expandeq}
	\begin{aligned}
	&\|(\Dev-\PiuCR)w\|^2_{0,K} = h^2F_1(w,K), \quad
	\|(\Dev-\PiuECR)w\|^2_{0,K}= h^2F_2(w,K),\\
	&
	(v-\PiCR v,\Pih v)_K=h^2F_3(v,K),
	\end{aligned}
	\end{equation}
where for any element $K$ and unit tangential vector $\bmt_j$ corresponding to edge $e_j$ of $K$,
\begin{equation}\label{FsigmaG}
\begin{aligned}
           &F_1(w,K)\!:=\!\sum_{i,j=1}^{8}\gamij\int_{K}\operatorname{D}_iw\operatorname{D}_jw \dx,
           ~F_2(w,K)\!:=\!\sum_{i,j=1}^{8}\etaij\int_{K}\operatorname{D}_iw\operatorname{D}_jw \dx,\\
           &F_3(v,K)\!:=\!\frac18\sum_{i=1}^3\zetai\int_K\left(\frac{\partial^2 v}{\partial\bmt_i^2}\right)^Tv\dx.
\end{aligned}
\end{equation} 
\end{lemma} 
\begin{proof}
By the scaling argument and a direct calculation, constant $\zetai$ has the same value on different elements of a uniform triangulation and is independent of $h$.
By the expansion \eqref{RTdof}, $\PiRT\phiRT^i=\sum_{j=1}^{3}a_{\rm RT}^{ij}\otimes(\bmx-\bmp_j)\mbox{ with } a_{\rm RT}^{ij}=\frac{1}{2\vert{K}\vert}\int_{\F_j}\phiRT^i\bmn_j\ds.$
For linear  $\phiRT^i$, constants
$
a^{ij}_{\rm RT} = \frac{\vert{\F_j}\vert}{2\vert{K}\vert}\phiRT^i(\mathbf{m}_j)\bmn_j
$
where $\mathbf{m}_j$ is the midpoint of $\F_j$, which  implies that constants $a^{ij}_{\rm RT}$ are independent of $h$ and take the same value on different elements. 
Recall $\PiuCR$ and $\PiuECR$ in \eqref{def:inter}. By $\Pih(\bmx-\MK)=0$,
\begin{equation}
\label{PiuCRexp}
\PiuCR\phiRT^i=\sum_{j=1}^3\Dev(a_{\rm RT}^{ij}\otimes(\MK-\bmp_j)), \ 
\PiuECR\phiRT^i = \PiuCR\phiRT^i + (I-\Pih)\PiRT\phiRT^i,
\end{equation}
where $(I-\Pih)\PiRT\phiRT^i=\sum_{j=1}^3a_{\rm RT}^{ij}\otimes(\bmx-\MK)$.
By the scaling argument and the fact that $\Pih\phiRT^i=0$, all the terms on the right-hand side of the equation below on uniform triangulations are independent of $h$. 
$$
\begin{aligned}
\gamij&=\frac1{h^2|K|}(\Dev\phiRT^i,\Dev\phiRT^j)_K + \frac1{h^2|K|}(\PiuCR\phiRT^i,\PiuCR\phiRT^j)_K,\\
\etaij&=\gamij + \frac1{h^2|K|}((I-\Pih)\PiRT\phiRT^i,(I-\Pih)\PiRT\phiRT^j)_K\\
& + \frac1{h^2|K|}(\Dev\phiRT^i, (I-\Pih)\PiRT\phiRT^j)_K+\frac1{h^2|K|}((I-\Pih)\PiRT\phiRT^i, \Dev\phiRT^j)_K.\\
\end{aligned}
$$
Thus, $\gamij$ and $\etaij$ stay the same on different elements of a uniform triangulation and are independent of $h$.


Since $P_1(K,\R^{2\times2})\!=\!P_0(K,\R^{2\times2})+\operatorname{span}\left\{\phiRT^1,\phiRT^2,\!\cdots\!,\phiRT^8\right\}$, and $\operatorname{D}_i\phiRT^j(x)\!=\!\delta_{ij}$,
there exists $c\in P_0(K,\R^{2\times2})$ such that $ w=c + \sum_{i=1}^{8}a_i\phiRT^i$ with $a_i=\operatorname{D}_iw$. By $(\Dev-\PiuCR)c=(\Dev-\PiuECR)c=0$,
$$
\begin{aligned}
&(\Dev \!-\! \PiuCR)w=\sum_{i=1}^{8}a_i (\Dev-\PiuCR)\phiRT^i,~
(\Dev \!- \! \PiuECR)w=\sum_{i=1}^{8}a_i (\Dev-\PiuECR)\phiRT^i.\\
\end{aligned}
$$
A combination of this and \eqref{gammij} leads to the first and second equation in  \eqref{P1expandeq}.
%
By \cite[(3.38)]{MR4350533}, $(I-\PiCR)v=-\frac18\sum_{i=1}^3|e_i|^2\frac{\partial^2 v}{\partial\bmt_j^2}\piCRi$. A combination of this and \eqref{gammij} leads to the last equation in  \eqref{P1expandeq},
which completes the proof.
\end{proof}

Let $\X$ be $\Rd$ or $\Rdd$.
For any function $\tau$ taking values in $\X$, and positive integer~$l$, define $\Pi_K^l\tau \in  \mathbb{P}_l(K,\X)$ as  in \cite{MR3106483}  by
$
	\int_K D^\alpha\Pi_K^l\tau dx=\int_K D^\alpha \tau dx\mbox{ with }\vert{\alpha}\vert\leq l.
$
Let $\Pi_h^l\tau|_K=\Pi_K^l\tau$.
There exists the following error estimate of the interpolation error
\begin{equation}\label{errorPikl}
\left|\left(I-\Pi_K^l\right) \tau\right|_{m, K} \lesssim h^{l-m+1}|\tau|_{l+1, K}, \quad \forall\ 0 \leq m \leq l.
\end{equation}
Following \cite{MR4456710}, define $\phi_\alpha=\frac{1}{\alpha !}\left(\bmx-\MK\right)^\alpha\in \mathbb{P}_{|\alpha|}(K)$ and 
$$
\tilde \phi_\alpha=\begin{cases}\phi_\alpha &|\alpha|\le 1,
\\
\phi_\alpha - \sum_{|\beta| \leq|\alpha|-2} C_\alpha^\beta \phi_\beta &|\alpha|\ge 2,
\end{cases}
$$
where $C_\alpha^\beta=\frac{1}{\alpha !|K|} \int_K D^\beta\left(\bmx-\MK\right)^\alpha \dx$. 
There exists the following lemma in \cite{MR4456710}.

\begin{lemma}\label{lm:expansion}
	For any nonnegative integer $l$, $\sigma \in H^l(K,\Rdd)$ and $u \in H^l(K,\Rd)$,
	\begin{equation}\label{defakalpha}
	\begin{aligned}
	\Pi_K^l \sigma &=\sum_{|\alpha| \leq l} \ba_K^\alpha \tilde \phi_\alpha, \quad 
		\Pi_K^l u &=\sum_{|\alpha| \leq l} \bsa_K^\alpha \tilde \phi_\alpha,
	\end{aligned}
	\end{equation}
where $(\ba_K^\alpha)_{ij}=\frac{1}{|K|} \int_K D^\alpha \sigma_{ij} \dx$ and $(\bsa_K^\alpha)_{i}=\frac{1}{|K|} \int_K D^\alpha u_j \dx$.
Moreover, for~$l\le~3$,
	$ 
	(I-\Pi_K^{l-1}) \Pi_K^l \sigma =\sum_{|\alpha|=l} \ba_K^\alpha \tilde \phi_\alpha,
	$
	and
	$
	(I-\Pi_K^{l-1}) \Pi_K^l u =\sum_{|\alpha|=l} \bsa_K^\alpha \tilde \phi_\alpha.
	$
\end{lemma}

Since $(\Dev\!-\!\PiuCR)\phi_{(0,0)}\!=\!(\Dev\!-\!\PiuECR)\phi_{(0,0)}=0$, and $(I-\PiCR)\phi_{\beta}=0$ for $|\beta|=1$.
Lemma \ref{lm:expansion} indicates that the interpolation errors of $\sigma$ and $u$ can be expressed in terms of the basis $\ba_K^\beta\phi_{\beta}$ and $\bsa^\beta_K\phi_\beta$ as below, respectively 
\setlength{\jot}{0.01ex}
\begin{equation}\label{RTmatrix1}
\begin{aligned}
&(\Dev-\PiuCR)\Pi_h^1\sigma=\sum_{|\beta|=1}(\Dev-\PiuCR)(\ba_K^\beta\phi_{\beta}), \\
&(\Dev-\PiuCR)(I-\Pi_h^1)\Pi_h^{2}\sigma=\sum_{|\alpha|=2}(\Dev-\PiuCR)(\ba_K^\alpha\phi_{\alpha}),
\\
&(\Dev-\PiuECR)\Pi_h^1\sigma=\sum_{|\beta|=1}(\Dev-\PiuECR)(\ba_K^\beta\phi_{\beta}),\\
&(\Dev-\PiuECR)(I-\Pi_h^1)\Pi_h^{2}\sigma=\sum_{|\alpha|=2}(\Dev-\PiuECR)(\ba_K^\alpha\phi_{\alpha}),\\
&(I-\PiCR)(I-\Pi_h^2)\Pi_h^3u=\sum_{|\alpha|=3}(I-\PiCR)(\bsa_K^\alpha\phi_\alpha).
\end{aligned}
\end{equation}
By the triangle inequality and estimates in \eqref{errorPikl}, the interpolation error terms in Lemma \ref{th22} and \ref{th21} can be rewritten in a compact form using the particular basis functions in Lemma \ref{lm:expansion}, namely  
\begin{equation}\label{Intmain2}
	\begin{aligned} 
               &\|(\Dev-\PiuCR)\sigma\|_{\Omega}^2 = h^2F_1(\sigma, \Omega) + 2\sum_{K\in\Th}\sum_{\vert{\alpha}\vert=2}\sum_{\vert{\beta}\vert=1} I_K^{\alpha\beta} +\mathcal{O}(h^4),\\
               & \|(\Dev-\PiuECR)\sigma\|_{\Omega}^2= h^2F_2(\sigma, \Omega) + 2\sum_{K\in\Th}\sum_{\vert{\alpha}\vert=2}\sum_{\vert{\beta}\vert=1} J_K^{\alpha\beta}+\mathcal{O}(h^4), \\              
               &((I-\PiCR) u, \Pih u)=h^2F_3(u, \Omega)+2\sum_{K\in\Th}\sum_{|\alpha|=3}\sum_{|\beta|=0}T_K^{\alpha\beta} +\mathcal{O}(h^4),
	\end{aligned}
\end{equation}
where  cross terms
$$
\begin{aligned}
I_K^{\alpha\beta}=&((\Dev-\PiuCR)(\ba_K^\beta\phi_{\beta}), (\Dev-\PiuCR) (\ba_K^\alpha\phi_{\alpha}))_K,
\\
J_K^{\alpha\beta}=&((\Dev-\PiuECR)(\ba_K^\beta\phi_{\beta}), (\Dev-\PiuECR)(\ba_K^\alpha\phi_{\alpha}))_K,
\\
T_K^{\alpha\beta}=&((I-\PiCR)(\bsa_K^\alpha\phi_\alpha), \bsa_K^\beta\phi_\beta)_K.
\end{aligned}
$$ 

\begin{lemma}\label{lm:RTmatrix}
For $\sigma \in H^2(K,\Rdd)$ and $u \in H^3(K,\Rd)$, it holds that
\setlength{\jot}{0.01ex}
$$
\begin{aligned}
 I_K^{\alpha\beta}=& (\ba^\alpha_K)^T\ba^\beta_K:\bc_K^{\alpha\beta} - \frac{1}{2} ((\ba^\alpha_K)^T\circ \ba^\beta_K): \bd_K^{\alpha\beta}, 
\\
J_K^{\alpha\beta}=& (\ba^\alpha_K)^T\ba^\beta_K: \be_K^{\alpha\beta} 
 - \frac{1}{2} ((\ba^\alpha_K)^T\circ \ba^\beta_K):  \Bf_K^{\alpha\beta},
\qquad
T_K^{\alpha\beta}=(\bsa^\alpha_K)^T\bsa^\beta_Kc_K^{\alpha\beta},
\end{aligned} 
$$
where  $c^{\alpha\beta}_K=\int_K((I-\PiCR)\phi_\alpha)\phi_\beta\dx$, 
$\bc_K^{\alpha\beta} \!=  \!\int_K R_\alpha^2 (R_\beta^2)^T \dx$, 
$\bd_K^{\alpha\beta}\!=\!\int_K R_\alpha^2 \circ (R_\beta^2)^T \dx$,  
$\be^{\alpha\beta}_K\!=\!\int_K R_\alpha^2 (R_\beta^2)^T + R_\alpha^1 (R_\beta^2)^T \dx + R_\alpha^2(R_\beta^1)^T + R_\alpha^1(R_\beta^1)^T$, 
$\Bf^{\alpha\beta}_K\!=\!\int_K R_\alpha^2 \circ (R_\beta^2)^T + R_\alpha^1\circ (R_\beta^2)^T
+ R_\alpha^2\circ(R_\beta^1)^T\dx$ with 
$R_\alpha^1=(I  \!- \! \Pih)\PiRT(\phi_\alpha\bid)$ and $R_\alpha^2=(I-\Pih\PiRT)(\phi_\alpha\bid)$.


\end{lemma} 
\begin{proof}
Since $(I-\PiCR)(\bsa_K^\alpha\phi_\alpha)=\bsa_K^\alpha(I-\PiCR)\phi_\alpha$, this and the definition of $c_K^{\alpha\beta}$ lead to the last equation in Lemma 3.6.

Note that $\ba_K^\alpha \phi_\alpha \bmn =\ba_K^\alpha ((\phi_\alpha \bid)\bmn)$ and $\ba_K^\alpha\in \mathbb{P}_0(K,\Rdd)$. By \eqref{RTdof},
$$
\PiRT (\ba_K^\alpha \phi_\alpha)
=\ba_K^\alpha\sum_{i=1}^d\frac{1}{2\left|{K}\right|}\int_{\F_i} (\phi_\alpha \bid)\bmn_i \ds\otimes\left(\bmx-\bmp_i\right)
=\ba_K^\alpha \PiRT (\phi_\alpha\bid).
$$
Thus, by the definition of $R_\alpha^i$ above,
\begin{equation*}
\begin{aligned}
(\Dev-\PiuECR)(\ba_K^\alpha \phi_\alpha) + \ba_K^\alpha R_\alpha^1= (\Dev-\PiuCR)(\ba_K^\alpha \phi_\alpha)=\Dev(\ba_K^\alpha R_\alpha^2).
\end{aligned}
\end{equation*}
For any matrices $A$, $B$, $C$, $D\in \Rdd$, it holds that $\Dev(AB):\Dev(CD)=(AB):(CD) - \frac12\Tr(AB)\Tr(CD)$, $(AB):(CD)=(C^TA):(DB^T)$, and $\Tr(AB)\Tr(CD)=(C^T\circ A):(D\circ B^T)$.
Let $A= \ba_K^\beta$, $B=R_\beta^2$, $C=\ba_K^\alpha$, and $D=R_\alpha^2$. Then
\begin{equation*}
\begin{aligned}
I_K^{\alpha\beta}
=&\int_K (\ba_K^\alpha)^T(\ba_K^\beta) : R_\alpha^2(R_\beta^2)^T\dx
-\frac12\int_K(\ba_K^\alpha)^T\circ\ba_K^\beta : R_\alpha^2 \circ (R_\beta^2)^T\dx,
\\
J_K^{\alpha\beta}
= & I_K^{\alpha\beta} 
+\int_K(\ba_K^\alpha)^T\ba_K^\beta :  
(R_\alpha^1(R_\beta^2)^T   
+ R_\alpha^2(R_\beta^1)^T
+ R_\alpha^1(R_\beta^1)^T)\dx
\\
&-\frac12\int_K(\ba_K^\alpha)^T\circ(\ba_K^\beta) : (R_\alpha^1\circ (R_\beta^2)^T
+ R_\alpha^2\circ(R_\beta^1)^T)\dx,
\end{aligned}
\end{equation*}
which completes the proof of the first two equations in Lemma \ref{lm:RTmatrix}, which completes the proof.  \end{proof}


\begin{lemma}\label{lm25}
It holds on uniform triangulations that
\begin{equation}\label{ASleqh4}
\begin{aligned}
\sum_{K\in\Th}(\sum_{\vert{\alpha}\vert=2}\sum_{\vert{\beta}\vert=1} (I_K^{\alpha\beta} + J_K^{\alpha\beta}) + \sum_{|\alpha|=3}\sum_{|\beta|=0} T_K^{\alpha\beta})&\lesssim h^4(\Vert{\sigma}\Vert_{3,\Omega}^2+\|u\|_{4,\Omega}^2),
	\end{aligned}
\end{equation}
	provided that $\sigma\in H^{3}(\Omega,\Rdd)$, and $u\in H^4(\Omega,\Rd)$.
\end{lemma}
\begin{proof}
The partition $\mathcal{T}_h$ of domain $\Omega$ includes the set of parallelograms $\mathcal{N}_1$ and that of a few remaining boundary triangles $\mathcal{N}_2$. Let
$
	\kappa=\left|\mathcal{N}_2\right|
$
denote the number of the elements in $\mathcal{N}_2$. 

Consider $\sum_{K\subset \mathcal{N}_1}I_K^{\alpha\beta}$. Let $\mathbf{c}$ be the centroid of a parallelogram formed by two
	adjacent elements $K_1$ and $K_2$. Define a mapping $T: \bmx\to\tilde{\bmx}=2\mathbf{c}-\bmx$. It is obvious that $T$ maps $K_1$ onto $K_2$ and $\bmx-\mathbf{M}_{K_1}=-(\tilde{\bmx}-\mathbf{M}_{K_2})$. 
Similar to the analysis of Lemma~3.5 in \cite{MR4456710}, it follows from $|\alpha|=2$, $|\beta|=1$ and the geometric symmetry that	
$$
\bc_{K_2}^{\alpha\beta}=-\bc_{K_1}^{\alpha\beta},\quad 
\bd_{K_2}^{\alpha\beta}=-\bd_{K_1}^{\alpha\beta},\quad
\be_{K_2}^{\alpha\beta}=-\be_{K_1}^{\alpha\beta},\quad
\Bf_{K_2}^{\alpha\beta}=-\Bf_{K_1}^{\alpha\beta}.
$$
There follows the decomposition below, which also works for $ J_{K}^{\alpha\beta}$ and $ T_{K}^{\alpha\beta}$,
	\begin{equation}\label{expanAPi1S}
		\begin{aligned}
 I_{K_1}^{\alpha\beta} +  I_{K_2}^{\alpha\beta}&  =
 (\ba^\alpha_{K_1}-\ba^\alpha_{K_2})^T\ba^\beta_{K_1}:\bc_{K_1}^{\alpha\beta} 
 +(\ba^\alpha_{K_2})^T(\ba^\beta_{K_1}-\ba^\beta_{K_2}):\bc_{K_1}^{\alpha\beta} 
 \\
&- \frac{1}{2} ((\ba^\alpha_{K_1} -\ba^\alpha_{K_2})^T\circ \ba^\beta_{K_1}):\bd_{K_1}^{\alpha\beta} 
\\
&- \frac{1}{2}((\ba^\alpha_{K_2})^T\circ (\ba^\beta_{K_1}-\ba^\beta_{K_2})):\bd_{K_1}^{\alpha\beta}
+ \mathcal{O}(h^4\Vert{\sigma}\Vert_{3,K_1\cup K_2}^2).
		\end{aligned}
	\end{equation}
Note that  
$
\frac{1}{\left|K_1\right|} \int_{K_1} \operatorname{D}^\gamma \sigma d x-\frac{1}{\left|K_2\right|} \int_{K_2} \operatorname{D}^\gamma \sigma \dx=~0
$
for all $\sigma\in \bbP_{|\gamma|}(K_1\cup K_2,\Rdd)$. 
By Bramble–Hilbert lemma, for any $|\alpha|=2$ and $|\beta|=1$,
\begin{equation}\label{errorak}
 |(\ba^\alpha_{K_1}-\ba^\alpha_{K_2})_{ij}|+|(\ba^\beta_{K_1}-\ba^\beta_{K_2})_{ij}| \lesssim \|\sigma_{ij}\|_{3, K_1 \cup K_2}.
\end{equation}
Recall $\ba_K^\gamma=\frac{1}{|K|} \int_K  \operatorname{D}^\gamma \sigma \dx$ in \eqref{defakalpha}. Thus, for any $1\le i, j\le 2$,
\begin{equation}\label{estaK}
	\vert{(\ba_{K}^\gamma)_{ij}}\vert\leq \frac{1}{\vert{K}\vert}\left\|{ \operatorname{D}^\gamma\sigma_{ij}}\right\|_{0,K}\Vert{1}\Vert_{0,K}\lesssim h^{-1}\Vert{\sigma_{ij}}\Vert_{|\gamma|,K}.
\end{equation}
By the definition of $\phi_\alpha$, it holds that
\begin{equation}
\begin{aligned}\label{estcd}
|{\bc_{K}^{\alpha\beta}}| + |{\bd_{K}^{\alpha\beta}}| + |{\be_{K}^{\alpha\beta}}| + |{\Bf_{K}^{\alpha\beta}}|&\lesssim h^3|K|.
\end{aligned}
\end{equation}
A substitution of \eqref{errorak}, \eqref{estaK} and \eqref{estcd} to \eqref{expanAPi1S} gives
\begin{equation}\label{impIN1}
\begin{aligned}
\sum_{K\subset \mathcal{N}_1} I_K^{\alpha\beta}
\lesssim h^4\sum_{K\in\mathcal{N}_1}\left\|{\sigma}\right\|_{3,K}^2 \lesssim h^4\left\|{\sigma}\right\|^2_{3,\Omega}
\end{aligned}
\end{equation}
Note that $\mathcal{N}_2\subset\partial_h\Omega:=\{\bmx \in \Omega: \exists \mathbf{y} \in \partial \Omega \mbox { such that } \operatorname{dist}(\bmx, \mathbf{y}) \leq h\}$, it follows from \cite[Lemma 3.7]{MR4350533}, the Cauchy-Schwarz inequality, \eqref{estaK} and   \eqref{estcd} that
\begin{equation*}
\begin{aligned}
|{\sum_{K\subset\mathcal{N}_2} I_K^{\alpha\beta}}|&\lesssim h^3(\sum_{K\in\mathcal{N}_2}|{\sigma}|_{1,K}^2)^{\frac{1}{2}}(\sum_{K\in\mathcal{N}_2}|{\sigma}|_{2,K}^2)^{\frac{1}{2}} 
\lesssim  h^4\|\sigma\|_{3,\Omega}^2.
\end{aligned}
\end{equation*}
A combination of this and \eqref{impIN1} leads to the estimate of $I_K^{\alpha\beta}$ in \eqref{ASleqh4}.
Similar analysis can yield the other estimates in \eqref{ASleqh4}, which completes the proof.
\end{proof}

Thanks to Lemma \ref{lm25}  and \eqref{Intmain2}, there exists the following fourth order accurate expansion of $\|\nabla u-\PiuCR\sigma\|^2_{0,\Omega}$, $(u-\PiCR u, u)$ and $\|\nabla u-\PiuECR\sigma\|^2_{0,\Omega}$.
\begin{lemma}\label{DevPiRTs}
It holds for $\sigma \!\in\! H^{3}(\Omega,\Rdd)$ and $u\! \in \!H^{4}(\Omega,\Rd)$ with $\nabla u \!=\!\Dev \sigma$ that
\begin{equation*}
\begin{aligned}
&\left\Vert{\nabla u-\PiuCR\sigma}\right\Vert^2 = h^2F_1(\sigma,\Omega) + \mathcal{O}(h^4(\Vert{\sigma}\Vert_{3,\Omega}^2+\|u\|_{4,\Omega}^2)), \\
&\left\Vert{\nabla u-\PiuECR\sigma}\right\Vert^2 = h^2F_2(\sigma,\Omega) + \mathcal{O}(h^4(\Vert{\sigma}\Vert_{3,\Omega}^2+\|u\|_{4,\Omega}^2)),\\
&(u-\PiCR u, u) = h^2F_3(u,\Omega) + \mathcal{O}(h^4(\Vert{\sigma}\Vert_{3,\Omega}^2+\|u\|_{4,\Omega}^2)),
\end{aligned}
\end{equation*}
where, for $i=1,2,3$, $F_i(\cdot,\Omega):=\sum_{K\in\Th}F_i(\cdot,K)$, and $F_i(\cdot,K)$ is defined in \eqref{FsigmaG}.
\end{lemma}

\subsection{\!Asymptotic expansions of eigenvalues and extrapolation algorithm}

A combination of Lemma \ref{th22}, \ref{th21} and \ref{DevPiRTs} yields the following asymptotic expansions of eigenvalues by the CR element  and the ECR element.

\begin{theorem}\label{CRexpthem}
Let $(\lambda, \bmu, p)$, $(\CRlam, \CRu, \CRp)$ and $(\ECRlam, \ECRu, \ECRp)$ be the solutions of \eqref{sto1} and \eqref{sto3} by the CR element and the ECR element, respectively, and $\sigma = \nabla \bmu + p\bid$. Assume that $(\sigma,u)\in H^3(\Omega,\Rdd)\times H^4(\Omega,\Rd)$. On a uniform triangulation with sufficiently small $h$, it holds that
$$
\begin{aligned}
\lambda-\CRlam&=h^2(F_1(\sigma,\Omega)- 2 \lambda F_3(u,\Omega)) +\mathcal{O}(h^4|\ln h|(\Vert{\sigma}\Vert_{3,\Omega}^2+\|u\|_{4,\Omega}^2)),
\\
\lambda-\ECRlam&=h^2 F_2(\sigma,\Omega) +\mathcal{O}(h^4|\ln h|(\Vert{\sigma}\Vert_{3,\Omega}^2+\|u\|_{4,\Omega}^2)).
\end{aligned}
$$
\end{theorem}

Denote the approximate eigenvalues on $\Th$ by $\lambda_h$, which are either $\CRlam$ by the CR element or $\ECRlam$ by the ECR element. Define  eigenvalues by extrapolation algorithm
\begin{equation}\label{(3.58)}
\lambda^{\mathrm{EXP}}_h=\frac{4 \lambda_h- \lambda_{2 h}}{3}.
\end{equation}
Since $F_1(\sigma,\Omega)$, $F_2(\sigma,\Omega)$ and $F_3(u,\Omega)$  are independent of $h$ when the solution is smooth enough. A direct application of Theorem \ref{CRexpthem} proves the optimal convergence of eigenvalues by extrapolation algorithm as presented below.

 \begin{theorem}\label{ECRexpthem} 
If $\lambda$ is a simple eigenvalue, extrapolation eigenvalues converge at a higher rate 4, namely,
$$
\left|\lambda_{\mathrm{CR}}^{\mathrm{EXP}}-\lambda\right|  + \left|\lambda_{\mathrm{ECR}}^{\mathrm{EXP}}-\lambda\right| \lesssim h^4|\ln h| (\Vert{\sigma}\Vert_{3,\Omega}^2+\|u\|_{4,\Omega}^2),
$$
where $\CRlam^{\mathrm{EXP}}$ and $\ECRlam^{\mathrm{EXP}}$ are extrapolation eigenvalues in \eqref{(3.58)} by the CR element and the ECR element, respectively.
\end{theorem}

\begin{remark}
If $\lambda$ is a multiple eigenvalue, approximate eigenfunctions on triangulations with different mesh size may approximate to different functions in eigenfunction space $M(\lambda)$. Then, the asymptotic expansion of eigenvalues in Theorem \ref{CRexpthem} cannot lead to a theoretical estimate of extrapolation eigenvalues in \eqref{(3.58)}. Some numerical tests in Section 5 show that extrapolation algorithm can also improve the accuracy of multiple eigenvalues to $\mathcal{O}\left(h^4\right)$ if the eigenfunction is smooth enough.
\end{remark}

\section{Numerical examples}\label{Numerical}
In this section, we aim to verify the superconvergence of CR and ECR element for Stokes source problem, and the efficiency of extrapolation algorithm \eqref{(3.58)} for Stokes eigenvalue problem.

In the following numerical experiments,  the unit domain $\Omega:=(0,1)^2$ is always used. Initial triangulation  $\mathcal{T}_{1}$ consists of two isosceles right triangles.
And we refine triangulation $\mathcal{T}_{i}~(i \ge 1)$ into a half-sized triangulation uniformly to get $\mathcal{T}_{i+1}$. 
For brevity, we abbreviate the norm $\|\cdot\|_{0,\Omega}$ as $\|\cdot\|$  in the following. 


\subsection{Example 1}
	Consider Stokes source problem \eqref{source} with the source term $f$ determined by the  exact solution $\bmu^f \!= \!(-\frac{\partial \varphi}{\partial y},\frac{\partial \varphi}{\partial x})^T$, $ p^f \!=\! \cos(\pi x)\sin(\pi y) $ and $\varphi(x,y)=(\sin\pi x)^2(\sin\pi y)^2 \exp(x+2y)$.
Define $\|\cdot\|_f = \|\cdot\|/ \|f\|$, which is equivalent to $\|\cdot\|$ as $f$ is determined.   Figure~\ref{fig:5-1a} plots the error of the approximate pressure $\|p-\CRp^f\|_f$ and the velocity $\|\nabla_h(u-\CRu^f)\|_f$, the postprocessed pressure  $\|p-\Kh \CRp^f\|_f$ and the  postprocessed velocity  $\|\nabla_h(u-\Kh\CRu^f)\|_f$, and also the error of the approximate pressure and the velocity with respect to the corresponding pseudostress interpolation  $\|\CRp^f - \PiuCR \sigma^f\|_f$ and $\|\nabla_h u-\Pip\sigma^f)\|_f$ for the CR element, and Figure~\ref{fig:5-1b} plots those for the ECR element. The numerical results in Figure~\ref{fig5.1} coincide with the optimal supercloseness and optimal superconvergence  in Theorem~\ref{th:superclose} and Theorem~\ref{th:superconvergence}, respectively.

\begin{figure}[tbhp]
	\centering
 	\subfloat[CR element]{\label{fig:5-1a}\includegraphics[width=6cm, height=5cm]{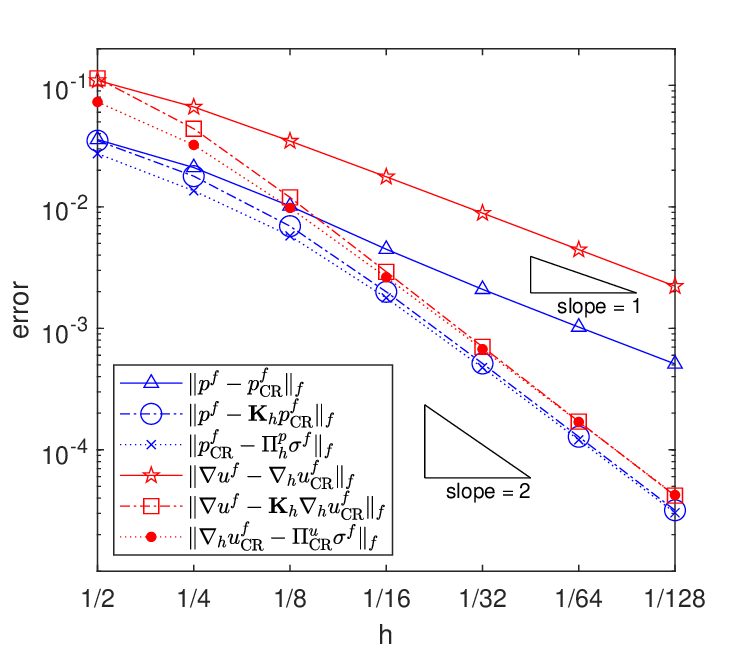}}
	\subfloat[ECR element]{\label{fig:5-1b}\includegraphics[width=6cm, height=5cm]{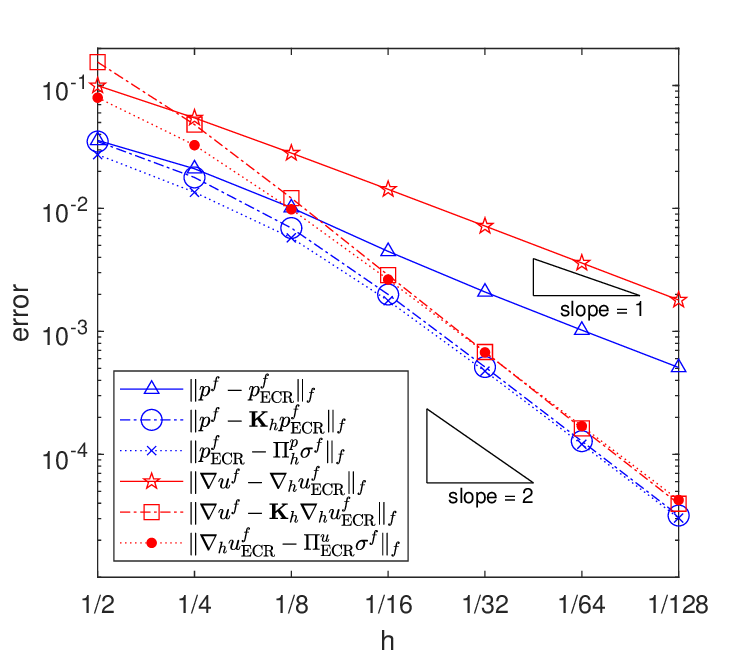}}
	\caption{Errors of CR element and ECR element for the Stokes source problem.}
	\label{fig5.1}
\end{figure}

\subsection{Example 2}
Consider the stokes eigenvalue problem \eqref{sto1} on unit domain. Since the exact eigenvalues are not known, we use element pair $(P_3,P_2)$ on $\mathcal{T}_{10}~(h=1/512)$ to compute the reference eigenvalues $\lambda_1 = 52.344691169,$ $\lambda_2 = 92.124393972 $ and
$\lambda_3 =  92.124393972,$ where $\lambda_2 = \lambda_3$ are multiple eigenvalues.
The relative errors of the first approximate eigenvalues of CR element, ECR element and their corresponding extrapolation eigenvalues on uniform triangulations $\mathcal{T}_i$ are presented in Table~\ref{tab6.1}, and the error results of multiple eigenvalues are shown in Figure~\ref{fig6.2}.
Table~\ref{tab6.1} and Figure~\ref{fig6.2} show that extrapolation algorithm for the CR element and ECR element improves the convergence rate of both single eigenvalues and multiple eigenvalues to a higher rate 4. The  numerical observation above directly verify the conclusions in Theorem~\ref{ECRexpthem} and its remark.
	\begin{table}[!ht]
	\begin{center}
	\caption{The relative errors of the first eigenvalues and the extrapolation eigenvalues.}
	\label{tab6.1}       
	\scalebox{0.95}{\begin{tabular}{lllllllll}
			\hline\noalign{\smallskip}
			& $\lambda_{\mathrm{CR},1} $ &rate & $\lambda_{\mathrm{CR},1}^{\mathrm{EXP}} $ &rate  & $\lambda_{\mathrm{ECR},1} $ &rate & $\lambda_{\mathrm{ECR},1}^{\mathrm{EXP}} $  &rate \\
			\noalign{\smallskip}\hline\noalign{\smallskip}
			$\mathcal{T}_3$  &1.1808e-01 &--   &-- 		    &--   &3.5002e-01 &--   &--         &--  	\\
			$\mathcal{T}_4$  &3.2962e-02 &1.84 &4.5893e-03  &--   &9.0697e-02 &1.95 &4.2548e-03 &--  	\\
			$\mathcal{T}_5$  &8.9271e-03 &1.88 &9.1544e-04  &2.33 &2.2723e-02 &2.00 &6.4759e-05 &6.04  \\
			$\mathcal{T}_6$  &2.3015e-03 &1.96 &9.3015e-05  &3.30 &5.7247e-03 &1.99 &5.8595e-05 &0.14  \\
			$\mathcal{T}_7$  &5.8063e-04 &1.99 &6.9912e-06  &3.73 &1.4354e-03 &2.00 &5.6166e-06 &3.38  \\
			$\mathcal{T}_8$  &1.4551e-04 &2.00 &4.7022e-07  &3.89 &3.5915e-04 &2.00 &4.0567e-07 &3.79  \\
			$\mathcal{T}_9$  &3.6400e-05 &2.00 &3.0000e-08  &3.97 &8.9807e-05 &2.00 &2.6000e-08 &3.96  \\
			\hline\noalign{\smallskip}
	\end{tabular}}
	\end{center}
\end{table}

\begin{figure}[!ht]
	\centering
	\subfloat[CR element]{\label{fig:5-2a}\includegraphics[width=6cm, height=5cm]{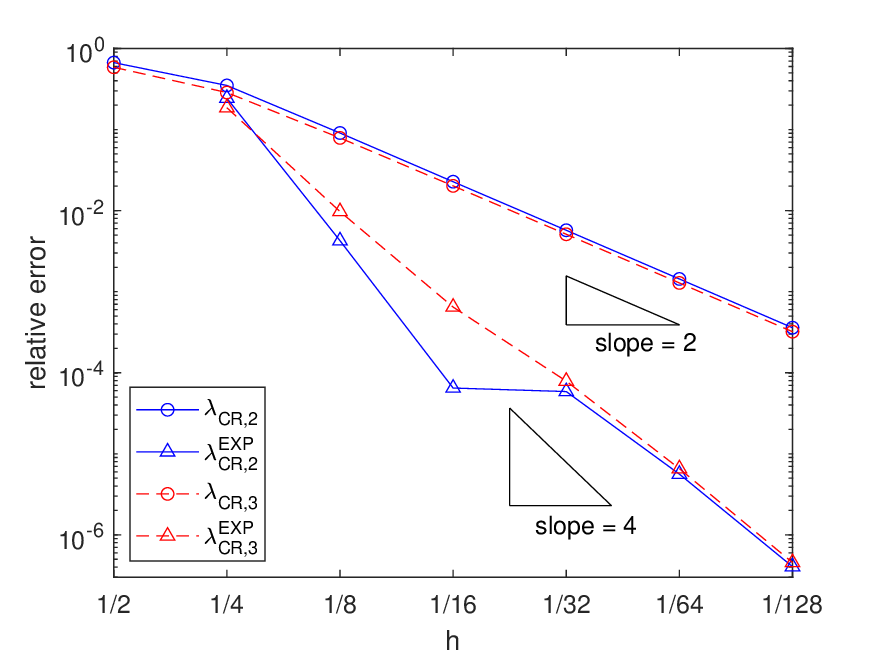}}
	\subfloat[ECR element]{\label{fig:5-2b}\includegraphics[width=6cm, height=5cm]{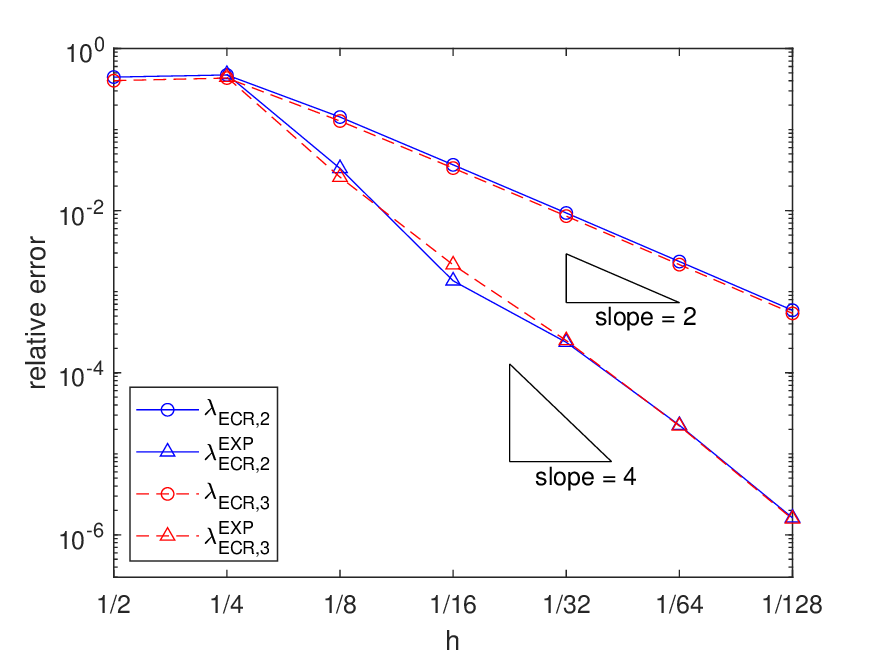}}
	\caption{The relative errors of the multiple eigenvalues from CR element and ECR element.}
	\label{fig6.2}
\end{figure}
%

\section*{Acknowledgments}
The authors are greatly indebted to Professor Jun Hu from Peking University for many useful discussions and the guidance.  

\bibliographystyle{plain} 
\bibliography{sample}
\end{document}